\documentclass[12pt, twoside, leqno]{article}

\usepackage{amsmath,amsthm}
\usepackage{amssymb}
\usepackage{enumerate}
\usepackage[all]{xy}
\usepackage{nth}
\usepackage{graphicx}
\usepackage{multirow}
\usepackage{url}
\usepackage{wrapfig}

\newtheorem{thm}{Theorem}[section]

\newtheorem{prop}[thm]{Proposition}

\newenvironment{proof_cor}{\paragraph{\textnormal{\emph{Proof of correctness.}}}}{\qed}

\theoremstyle{definition}
\newtheorem{alg}[thm]{Algorithm}
\newtheorem{defin}[thm]{Definition}
\newtheorem{rem}[thm]{Remark}

\numberwithin{equation}{section}

\newcommand{\RZ}{\mathbb Z}
\newcommand{\FQ}{\mathbb Q}
\newcommand{\iquat}{\mathcal H}
\newcommand{\QF}[1] {\FQ(\sqrt{#1})}
\newcommand{\Rideal}[1] {\langle#1\rangle_R}
\newcommand{\Lideal}[1] {\langle#1\rangle_L}
\newcommand{\ord}{\mathcal O}

\begin{document}


\baselineskip=17pt


\title{Ideal Class Group Algorithms\\in the Ring of Integral Quaternions}

\author{Anton S. Mosunov\\
St. Petersburg National Research University\\ 
of Information Technologies, Mechanics and Optics\\
49 Kronverkskiy Prospect (St. Petersburg, Russia)\\
E-mail: antonmosunov@gmail.com}
\date{}

\maketitle


\renewcommand{\thefootnote}{}



\renewcommand{\thefootnote}{\arabic{footnote}}
\setcounter{footnote}{0}

\begin{abstract}
An ideal is a classical object of study in the field of algebraic number theory. In maximal quadratic orders of number fields, ideals usually represented by the $\RZ$-basis. This form of representation is used in most of the algorithms for ideal manipulation. However, this is not the only option, as an ideal can also be represented by its generating set. Moreover, in certain quadratic orders of the ring of integral quaternions, every ideal can be represented by a single quaternion, called a \emph{pseudo generator}.

\par In this paper, we develop algorithms which allow us to manipulate ideals using solely their pseudo generators. We also demonstrate the connection between the three squares problem and factorization. This result comes as an extension of a number factoring technique discovered by Fermat, which allows the factoring of an integer if a pair of its two squares representations is known. In addition to this, we present a number of identities, which ideals must satisfy; some computational data on the number of ambiguous classes, generated by the equation $\rho\mu = -\mu\rho$; and a new approach to the problem of finding a non-trivial divisor of a class number.
\end{abstract}

\section{Introduction}

\par In the field of algebraic number theory, an integral two-sided ideal is well known, and has been thoroughly studied over the past two centuries. In \emph{maximal quadratic orders} $\ord$ of \emph{number fields}, $\QF N$, where $N$ is a non-zero squarefree integer, there are many things known about ideals. Namely, we know that equivalence classes of ideals in $\ord$ form a group, known as the \emph{ideal class group}. If we consider this group, we know how to manipulate its elements. But there are still many more things that we don't know about ideals. For example, even though we know that the size of the ideal class group is finite, we do not have any effective algorithm for its computation. There are also various heuristics known, such as the Cohen-Lenstra heuristics \cite{cohen1, cohen2}, which give certain predictions regarding the structure of an odd part of the ideal class group. However, we still don't have an efficient algorithm to compute its structure.

\par In this paper, we suggest studying an ideal class group from a different perspective. When $N < 0$, instead of $\QF N$, we consider it in the \emph{ring of integral quaternions} $\iquat$. The reason is that $\iquat$ has many beautiful properties that $\QF N$ may not have. For example, $\iquat$ is a \emph{principal ideal domain}, which means that each ideal of $\iquat$ can be represented by a single generator. Because of this fundamental property, even though $\ord(\mu)$, a quadratic order of $\iquat$, isomorphic to $\ord$, is not a principal ideal domain, we can still represent each ideal in $\ord(\mu)$ with a unique (up to a multiplication by a unit) integral quaternion, which we call a \emph{pseudo generator}. Knowing this, we asked the following question: is it possible to take advantage of this fact, and come up with algorithms, that manipulate ideals using only their pseudo generators? Eventually it was discovered that it is possible, and we introduce these algorithms in the section \ref{sec:algorithms}.

\par Another interesting property of the ring of integral quaternions is that, unlike $\ord$ in $\QF N$, there are many quadratic orders in $\iquat$, which are isomorphic to $\ord$. The total number of quadratic orders of a fixed discriminant $\Delta$ is related to the number of representations of an integer $|\Delta|$ as a sum of three squares. In the first half of the 20th century, mathematicians Hurwitz and Venkov explored the connections, that arise between various quadratic orders in $\iquat$. In continuation of this thread, in the section \ref{sec:fermat} we demonstrate how a pair of quadratic orders can generate a non-trivial factor of $N$. In the same section, we demonstrate that the Fermat number factoring technique is a special case of the theory, introduced in this paper.

\par In addition to the two sections mentioned above, the paper is structured as follows: section 2 describes the previous work in the field. Section 3 gives background information on the theory of integral quaternions and quadratic orders in $\iquat$, which was developed by Hurwitz and Venkov. Section 6 describes three identities that ideals must satisfy. Section 7 presents some computational data regarding the number of ambiguous classes within an ideal class group, which were generated by the equation $\rho\mu = -\mu\rho$. Section 8 introduces a novel approach to the problem of finding a divisor of a class number. Finally, section 9 concludes the paper.

\section{Previous Work}

The theory of integral quaternions was first introduced by Hurwitz \cite{hurwitz}. After Hurwitz, it was the soviet mathematician Boris Venkov, who expanded the theory by studying ideals in $\iquat$ \cite{ven1, ven2, ven3, ven4}. The monographs of Voight \cite{voight} and Vign\'eras \cite{vigneras} give a thorough explanation on how ideals arise in the context of quaternion algebras.

\par In terms of algorithms, there is a \texttt{Sage} code written by D. R. Kohel in 2005 \cite{kohel}, and maintained by J. Bobber and W. Stein \cite{bobber}, which has a broad functionality on quaternion algebras, including the operations on ideals. However, this code utilizes a different and more generalized approach to ideal manipulation. In this paper, we restrict our attention to the Hamilton algebra, and we tend to operate ideals using \emph{only} their pseudo generators, without ever mentioning a $\RZ$-basis. Among our algorithms, we also describe an algorithm on \emph{ideal reduction}, which, to our best knowledge, is not present in any literature available.

\par As for other algorithms, we would like to mention a paper by Kirschmer and Voight \cite{kirschmer}. Though not related to our research, it demonstrates another application of ideals in quaternion algebras (namely algebras defined over number fields). This paper describes algorithms, which count and enumerate representatives of the ideal classes of an Eichler order in a quaternion algebra defined over a number field.

\section{Background} \label{sec:background}

In this paper, we assume that our reader is familiar with some fundamental notions of algebraic number theory, such as an \emph{imaginary number field} $\QF{-m}$ (for $m$ positive and squarefree), a \emph{maximal quadratic order} $\ord$, and a \emph{(two-sided) $\ord$-ideal}. We encourage our reader to refer to the monographs \cite[Chap. 4]{jacobson} and \cite[Sec. 1.6]{mollin} for more details.

\par In this section, we shall primarily focus on the Hurwitz theory and the connection that exists between quadratic orders in number fields and quadratic orders in the ring of integral quaternions.

Consider the Hamilton algebra, where the trio of imaginary units $i$, $j$, $k$ produces the following multiplication table:

$$
\begin{array}{c | c | c | c }
& i & j & k\\
\hline
i & -1 & k & -j\\
\hline
j & -k & -1 & i\\
\hline
k & j & -i & -1
\end{array}
$$

\par The elements of Hamilton algebra are called \emph{quaternions}, and have the form $q = t + xi + yj + zk$, where $t$, $x$, $y$, $z$ are real numbers. We shall use the standard denotations of Venkov \cite{ven1}

\arraycolsep=1.4pt\def\arraystretch{2}

$$
\begin{array}{cc}
\Re(q) = t; & N(q) = t^2 + x^2 + y^2 + z^2;\\
\overline{q} = t - xi - yj - zk; & q^{-1} = \overline{q}/N(q).
\end{array}
$$

to refer to the \emph{real part}, \emph{norm}, \emph{conjugate} and the \emph{inverse} of a quaternion. We shall also write $\vec q = xi + yj + zk$ to denote the \emph{vector part} of a quaternion $q$.

\begin{defin}
An \emph{integral quaternion} is a quaternion of the form

$$
\frac{a + bi + cj + dk}{2},
$$

where $a$, $b$, $c$, $d$ are integers that satisfy the congruence $a \equiv b \equiv c \equiv d$ (mod 2). We shall denote the set of integral quaternions by $\iquat$.
\end{defin}

\begin{defin}
An integral quaternion is \emph{primitive}, if there is no positive integer other than 1, which divides all of its coefficients $a$, $b$, $c$, $d$.
\end{defin}

Hurwitz demonstrated that $\iquat$ is a maximal non-commutative ring with identity \cite{hurwitz}. In this ring, there exist 24 units, namely

$$
\pm 1, \,\,\, \pm i, \,\,\,\pm j, \,\,\,\pm k, \,\,\,\frac{\pm 1 \pm i \pm j \pm k}{2}.
$$

Furthermore, $\iquat$ is a \emph{euclidean domain} \cite[\textsection 4, a), c)]{ven1}, and therefore for any pair of integral quaternions $q$, $r$ there exists a unique (up to a multiplication by a unit from the left (or right)) integral quaternion, which we shall denote as $\gcd_r(q,r)$ (or $\gcd_l(q,r)$), of the largest norm, which divides both $q$ and $r$ from the right (or left).

In his \nth{2} letter \cite{ven2}, Venkov formulates the connection between quadratic orders in the ring of integral quaternions and quadratic orders of imaginary number fields. Following his reasoning, we first notice that each $q \in \iquat$ satisfies the relation

\begin{equation} \label{eq:quat_relation}
q^2 - 2\Re(q)q + N(q) = 0,
\end{equation}

whence any integral quaternion $\mu$ of norm $m$ with a zero real part satisfies the equation $\mu^2 = -m$. Next, recall that quaternion $\mu$ commutes with any other quaternion of the form $X + Y\mu$, where $X, Y \in \RZ$ \cite[\textsection 2, a)]{ven2}. Therefore, the quadratic order $\ord(\mu) = [1, \omega]$, where

\arraycolsep=2pt\def\arraystretch{1}

\begin{equation} \label{eq:omega}
\omega = \frac{r - 1 + \mu}{r}, \,\,\, r=\left\{
\begin{array}{l l}
1, & m \not \equiv 3 \pmod 4,\\
2, & m \equiv 3 \pmod 4,
\end{array}
\right.
\end{equation}

is isomorphic to the maximal quadratic order $\ord$ of $\QF{-m}$ \cite[\textsection 2]{ven3}. It is important to emphasize that this relation between $\ord$ and $\ord(\mu)$ exists if and only if $m$ is not representable in the form $4^k(8n + 7)$ for some non-negative integers $k$, $n$, since there are no three squares representations for integers of this form \cite[Chap. 4, \textsection 1]{grosswald}. In the remaining parts of this paper, we shall consider only squarefree values of $m = N(\mu)$, so $m$ must satisfy the congruence $m \not \equiv 7$ (mod 8).

\par After the formation of a quadratic order $\ord(\mu)$, which is isomorphic to $\ord$, it is now possible to study various ideals that belong to $\ord(\mu)$\footnote{Note that all ideals in $\ord(\mu)$ are \emph{two-sided}, since all elements of the quadratic order commute to each other \cite[\textsection 1, a)]{ven2}.}. A natural question arises: if $\ord$ and $\ord(\mu)$ are isomorphic, why do we need to study $\ord(\mu)$? After all, it is much easier to manipulate algebraic integers than quaternions. There are two main reasonings that led to the study of quadratic orders and ideals in $\iquat$:

\begin{enumerate}[a)]
\item The ring $\iquat$ is a Euclidean domain, and therefore it is a \emph{principal ideal domain} \cite[\textsection 4, b)]{ven1}. As we are about to demonstrate, this property allows us to represent an ideal with a single right (left) \emph{pseudo generator} (to be defined), which is an integral quaternion, and is unique up to a multiplication by a unit from the left (right);

\item A maximal quadratic order $\ord$ in $\QF{-m}$ is unique, which is not true for $\iquat$. Unless $m \equiv 7$ (mod 8), there exist several solutions to the equation $\mu^2 = -m$, hence several quadratic orders, isomorphic to $\ord$. There exist a number of interesting relations between various quadratic orders, which were studied by Venkov (e.g. in \cite[\textsection 19]{ven4}), and which we will demonstrate in this paper.
\end{enumerate}

\par There are other interesting properties of ideals and quadratic orders, which arise in $\iquat$ \cite[\textsection 5, 6]{ven3}. Consider a primitive ideal $\mathfrak a$ of some fixed quadratic order $\ord(\mu)$, and let the $\RZ$-basis of $\mathfrak a$ be equal to $[a, b + \omega]$ for some integers $a$, $b$. In this case, the norm of $\mathfrak a$ is equal to $N(\mathfrak a) = a$ \cite[Def. 4.33, Prop. 4.23]{jacobson}. Since $\iquat$ is a Euclidean domain, there exist integral quaternions $\rho = \gcd_r(a, b + \omega)$, and $\rho' = \gcd_l(a, b + \omega)$. Because $\mathfrak a$ is primitive, $\rho$ and $\rho'$ must be primitive as well.

\par Following the proof of Venkov \cite[\textsection 14]{ven2}, we can demonstrate that $N(\mathfrak a) = N(\rho) = N(\rho')$. Since $a$ is divisible by $\rho$ from the right and $\rho$ is primitive, we have $a = a_0 N(\rho)$ for some positive integer $a_0$. We also have $b + \omega = \xi\rho$ for some integral quaternion $\xi$, and because $\rho$ is the greatest common divisor, quaternions $\xi$ and $a_0\overline{\rho}$ have no common divisors from the right. Since $N(b+\omega) = N(\xi)N(\rho)$ is divisible by $a = a_0N(\rho)$, $N(\xi)$ is divisible by $a_0$ \cite[Thm. 4.24]{jacobson}, and therefore $a_0 = 1$, $N(\rho) = a = N(\mathfrak a)$. Analogously, we can show that $N(\rho') = N(\mathfrak a)$.

\par Now, let us demonstrate that $\rho$ unambiguously determines an ideal $\mathfrak a$ of $\ord(\mu)$. Define a set $\Rideal \rho$ as follows:

$$
\Rideal \rho = \left\{ \xi\rho \,\,\, | \,\,\, \xi \in \iquat \,\,\,\textnormal{and}\,\,\, \xi\rho \in \ord(\mu)\right\}.
$$

We call $\rho$ a \emph{pseudo generator} of $\Rideal \rho$, and use the letter $R$ to indicate that every element of this set is divisible by $\rho$ from the right. From the definition of $\Rideal \rho$, it follows that

$$
\Rideal \rho = (\rho)_L \cap \ord(\mu),
$$

where $(\rho)_L = \left\{ \xi\rho \,\,\, | \,\,\, \xi \in \iquat\right\}$ is a left ideal, generated by $\rho$. Now, let us demonstrate that $\mathfrak a = \Rideal \rho$. First of all, notice that all the elements of $\mathfrak a = [a, b + \omega]$ are divisible by $\rho$ from the right, so $\mathfrak a \subseteq \Rideal \rho$. To show that the converse statement holds, once again, we follow the proof of Venkov \cite[\textsection 14]{ven2}. Pick $q \in \Rideal \rho$. Since $q \in \ord(\mu)$, we have $q = X + \omega Y$ for some integers $X$ and $Y$. Since $X + Y\omega - Y(b + \omega) = X - bY$ is divisible by $\rho$ from the right, we can put $X = bY + ca$ for some positive integer $c$, which means that

$$
X  + Y\omega = ca + Y(b + \omega).
$$

\par We conclude that $\mathfrak a = \Rideal \rho$. Analogously, we can define

$$
\Lideal{\rho'} = \left\{ \rho'\xi \,\,\, | \,\,\, \xi \in \iquat \,\,\,\textnormal{and}\,\,\, \rho'\xi \in \ord(\mu)\right\},
$$

and demonstrate that $\mathfrak a = \Lideal{\rho'}$. In the remaining parts of this paper, we call $\rho$ and $\rho'$ the \emph{right} and \emph{left pseudo generators} of $\mathfrak a = \Rideal \rho = \Lideal{\rho'}$, respectively. For obvious reasons, if $\mathfrak a$ is non-primitive, i.e. representable in the form $c[a, b + \omega]$ for some integer $c > 1$, then both $\rho$ and $\rho'$ will be non-primitive integral quaternions, and their greatest integral divisor will be exactly $c$.

\par The proof above allows us to make the following observation: even though an ideal $\mathfrak a$ in $\ord(\mu)$ may have multiple generators, it still has a unique right (left) pseudo generator. Note that the right (left) pseudo generator is unique up to a multiplication by a unit of $\iquat$ from the left (right), because $\Rideal{\varepsilon\rho} = \Rideal \rho$ for any unit $\varepsilon$ (analogous statement holds for $\Lideal{\rho'}$). Hereinafter, all our results will be demonstrated for right pseudo generators, but by analogy we can demonstrate that they hold for left pseudo generators as well.

\vspace{1cm}

\par We now move our attention to the next important topic. Let $\mathfrak a = [\omega_1, \omega_2]$ be an arbitrary ideal in $\ord(\mu)$. Since $\mathfrak a = \Rideal \rho$ for some integral quaternion $\rho$, there exist a pair of integral quaternions $\zeta_1$ and $\zeta_2$ s.t.

$$
\omega_1 = \zeta_1\rho, \,\,\, \omega_2 = \zeta_2\rho.
$$

\par Since $N(\mathfrak a) = N(\rho)$, we conclude that $\zeta_1$, $\zeta_2$ do not have a common divisor from the right, and therefore $\gcd_r(\omega_1\mu, \omega_2\mu) = \rho\mu$. On the other hand, since $\omega_1, \omega_2$ both belong to $\ord(\mu)$, they commute with $\mu$, whence the greatest common divisor of $\omega_1\mu = \mu\omega_1$ and $\omega_2\mu=\mu\omega_2$ is divisible by $\rho$ from the right. This allows us to conclude, that there exists an integral quaternion $\mu'$, which satisfies the relation

\begin{equation} \label{eq:rhomu}
\rho\mu = \mu'\rho.
\end{equation}

\par If we multiply the previous equation by $\mu$ from the right, it is not hard to notice that $\mu'^2 = -m$, so $\mu'$ forms another quadratic order $\ord(\mu')$ in $\iquat$. A similar reasoning can be applied towards the left pseudo generator of an ideal. Thus, if $\mathfrak a = \Rideal \rho = \Lideal{\rho'}$ belongs to the quadratic order $\ord(\mu)$, then the relations

$$
\mu' = \rho\mu\rho^{-1}, \,\,\, \mu'' = \rho'^{-1}\mu\rho'
$$

generate the \emph{right} and \emph{left orders} $\ord(\mu')$ and $\ord(\mu'')$, respectively. The relation (\ref{eq:rhomu}) describes a very important property of a pseudo generator $\rho$: while $[a, b + \omega] = \Rideal \rho$ belongs to the quadratic order $\ord(\mu)$, $[a, b + \omega'] = \Lideal \rho$ belongs to the quadratic order $\ord(\mu')$, where $\omega' = (r - 1 + \mu')/r$.

\par Let us now answer the following question: if we know the pseudo generators of an ideal $\mathfrak a = \Lideal{\rho'} = \Rideal \rho$ in $\ord(\mu)$, then how do pseudo generators of $\overline{\mathfrak a}$ (the conjugate of $\mathfrak a$) look like?

\begin{prop}
Let $\mathfrak a = \Lideal{\rho'} = \Rideal \rho$ belong to the quadratic order $\ord(\mu)$. Then $\overline{\mathfrak a} =\Lideal{\overline \rho} = \Rideal{\overline{\rho'}}$.
\end{prop}

\begin{proof}
Let $\mathfrak a = [a,b + \omega]$. Then the right and left pseudo generators of an ideal can be found from the following relations:

$$
\rho = \gcd\,\!_r(a, b + \omega), \,\,\, \rho' = \gcd\,\!_l(a, b + \omega).
$$

Let $b + \omega = \xi\rho$ for some integral quaternion $\xi$. Then $\overline{b + \omega} = \overline{\xi\rho} = \overline{\rho} \cdot \overline{\xi}$. Since $a = N(\rho)$, the left generator of $\overline{\mathfrak a} = \left[a, \overline{b + \omega}\right]$ is equal to

$$
\gcd\,\!_l\left(N(\rho), \overline{b + \omega}\right) = \gcd\,\!_l\left(\overline{\rho} \cdot \rho, \overline{\rho} \cdot \overline{\xi}\right) = \overline{\rho}.
$$

Analogously, we can find the right pseudo generator $\overline{\rho'}$ of $\overline{\mathfrak a}$ by noticing that $a = N(\rho')$.
\end{proof}

As a special case of this proposition, we notice that an \emph{amgiuous} ideal $\mathfrak a = \Rideal \rho$, i.e. the ideal with a property $\mathfrak a = \overline{\mathfrak a}$, has $\overline{\rho}$ as its left pseudo generator. We also emphasize another important property of ambiguous ideals in the following theorem.

\begin{thm} \label{thm:ambiguous}
If $\mathfrak a = \Rideal \rho$ is ambiguous in the quadratic order $\ord(\mu)$, then $\rho$ divides $(2/r)\mu$ from the right, where $r$ is defined as in relation (\ref{eq:omega}).
\end{thm}

\begin{proof}
Since $\mathfrak a = \overline{\mathfrak a}$, an ideal $\mathfrak a = [a, b + \omega]$ contains both $b + \omega$ and $\overline{b + \omega}$. Let $b + \omega = \xi\rho$, $\overline{b + \omega} = \xi'\rho$ for some $\xi, \xi' \in \iquat$. Then

$$
b + \omega - \overline{b + \omega} = b + \frac{r - 1 + \mu}{r} - \left(b + \frac{r - 1 - \mu}{r}\right) = (2/r)\mu = \xi\rho - \xi'\rho;
$$

Hence, $\xi - \xi' = (2/r)\mu\rho^{-1}$, i.e. the quaternion $(2/r)\mu\rho^{-1}$ is integral, which means that $\rho$ divides $(2/r)\mu$ from the right.
\end{proof}

We are aware that in a maximal quadratic order $\ord$ of  $\QF{-m}$, if $\mathfrak a$ is ambiguous, then $N(\mathfrak a)$ divides $(2/r)^2 m$ \cite[Thm. 7.11]{jacobson}. The theorem above demonstrates that an analogous statement holds for $\ord(\mu)$ as well, where $N(\mu) = m$. Moreover, we can notice that our theorem implies this result, for if $\rho$ divides $(2/r)\mu$ from the right, then $N(\rho) = N(\mathfrak a)$ divides $N\left((2/r)\mu\right) = (2/r)^2m$.

So far, we have seen many interesting properties that ideals and their pseudo generators have. Now that we know how ideals get represented in the ring of integral quaternions, it would be interesting to study how to operate them. A common approach to manipulate ideals is to use their $\RZ$-basis representation $[a, b + c\omega]$, where integers $a$, $b$ satisfy inequalities $a > 0$ and $0 \leq b < a$, and $c$ divides both $a$ and $b$. This approach is used in ideal reduction \cite[Sec. 5.1]{jacobson} and multiplication \cite[Sec. 5.4]{jacobson} algorithms. In the next section, we introduce a number of algorithms, which allow us to manipulate ideals in the ring of integral quaternions using solely their pseudo generators. While developing these algorithms, our primary goal was to avoid \emph{any} mentioning of the $\RZ$-basis. As a result, besides the reduction and multiplication algorithms, we introduce two additional ones: an algorithm to restore the $\RZ$-basis of an ideal by its pseudo generator, and an algorithm to compute the right pseudo generator of an ideal using its left pseudo generator.

\section{Algorithms} \label{sec:algorithms}

We begin our explorations with the series of definitions, where we consider two integral quaternions $q$ and $r$ of the form

$$
\begin{array}{l}
q = t_0 + x_0 i + y_0 j + z_0 k,\\
r = t_1 + x_1 i + y_1 j + z_1 k.
\end{array}
$$

\begin{defin}
A \emph{scalar product} of quaternions $q$ and $r$ is the value $(q,r) = x_0x_1 + y_0y_1 + z_0z_1$;
\end{defin}

\begin{defin}
A \emph{full scalar product} of quaternions $q$ and $r$ is the value $(q,r)_\times = t_0t_1 + x_0x_1 + y_0y_1 + z_0z_1 = \Re(q)\Re(r) + (q,r)$.
\end{defin}

\begin{defin}
A \emph{vector product} of quaternions $q$ and $r$ is the value $[q,r] = (y_0z_1 - y_1z_0)i + (x_1z_0 -x_0z_1)j + (x_0y_1 - x_1y_0)k$.
\end{defin}

It is not hard to verify that, in case if $\Re(q)=\Re(r)=0$, $q$ and $r$ satisfy the relation

\begin{equation} \label{eq:product}
qr = -(q,r) + [q,r].
\end{equation}

Finally, we shall introduce one more interesting relation, which holds for any integral quaternions $q$ and $r$:

\begin{equation} \label{eq:qrrq}
\frac{qr + rq}{2} = \Re(q)\Re(r) - (q,r) + \Re(q)\vec r + \Re(r)\vec q.
\end{equation}

This equality can be easily demonstrated using the properties of quaternion arithmetic. Both of the relations above will be used in the proofs that will be introduced in the following subsections.

\par The following algorithms allow us to manipulate ideals in the ring of integral quaternions. Note that, in order to migrate from an imaginary number field $\QF{-m}$ to the ring of integral quaternions, we need to generate one initial representation of an integer $m$ as a sum of three squares, i.e. $m = x^2 + y^2 + z^2$. Fortunately, there exists a fast algorithm, which computes integers $x$, $y$, $z$ with high probability in random polynomial time \cite[p. S254]{rabin}.

\subsection{Restoring the $\RZ$-basis}

Consider the following problem: given a primitive ideal $\Rideal \rho$ and the quadratic order $\ord(\mu)$ which it belongs to, find the $\RZ$-basis of $\Rideal \rho$. In other words, we need to find integers $a, b$ s.t. $[a, b + \omega] = \Rideal \rho$.

\par In the section \ref{sec:background}, we have demonstrated that $a = N(\mathfrak a) = N(\rho)$. Now it is left only to find $b$. One possible way to solve this problem is to find the square root of $-m$ mod $r^2a$, where $r$ is defined as in (\ref{eq:omega}), for if $[a, b + \omega]$ is a valid ideal, then $b$ has to satisfy the congruence

$$
(rb + r - 1)^2 \equiv -m \pmod{r^2 a}.
$$

There exists an effective algorithm \cite[Alg. 2.3.8]{crandall} with the worst case complexity $O(\ln^4 a)$, which computes a square root mod $a$ if $a$ is an odd prime. However, we would like to solve this problem for \emph{all} the possible values of $a$.

\par Before we introduce another algorithm, which solves the problem for any $a$, let us first prove a couple of important propositions, which shall be quite useful in our future explorations.

\begin{prop} \label{prop:equal_scalar_products_rho}
If real parts of quaternions $\mu$ and $\mu'$ are equal, then the equality $\rho\mu = \mu'\rho$ implies $(\mu,\rho) = (\mu', \rho)$.
\end{prop}

\begin{proof}
Let us write $\rho, \mu$ as the sum of their real and vector parts:

$$
(\Re(\rho) + \vec \rho)(\Re(\mu) + \vec \mu) = (\Re(\mu') + \vec{\mu'})(\Re(\rho) + \vec \rho);
$$

$$
\Re(\rho)\Re(\mu) + \Re(\rho)\vec \mu + \Re(\mu)\vec \rho + \vec \rho \vec \mu = \Re(\mu')\Re(\rho) + \Re(\rho)\vec{\mu'} + \Re(\mu')\vec \rho + \vec{\mu'}\vec \rho.
$$

Since $\Re(\mu) = \Re(\mu')$,

$$
\Re(\rho)\vec \mu + \vec \rho \vec \mu = \Re(\rho)\vec{\mu'} + \vec{\mu'}\vec \rho.
$$

If we now apply the relation (\ref{eq:product}), we shall get

$$
-(\rho,\mu) + [\rho,\mu] + \Re(\rho)\vec\mu = -(\mu',\rho) + [\mu',\rho] + \Re(\rho)\vec{\mu'}.
$$

By taking the real part on both sides of the equality above, we obtain the desirable result: $(\mu,\rho) = (\mu', \rho)$.
\end{proof}

\begin{prop} \label{prop:main_prop}
If a primitive ideal $[a, b + \omega] = \Rideal \rho$ belongs to the quadratic order $\ord(\mu)$, then

\begin{equation} \label{eq:main_relation_rho}
(\omega, \rho) +\Re(\rho)\left(b + \frac{r - 1}{r}\right) = \Re(\xi)N(\rho).
\end{equation}

where $\xi$ is an integral quaternion, which satisfies the relation $b + \omega = \xi\rho$, and an integer $r$ is defined as in (\ref{eq:omega}).
\end{prop}

\begin{proof}
Since $b + \omega$ belongs to $\Rideal \rho$, there exists an integral quaternion $\xi$ s.t. $b + \omega = \xi\rho$. Because $\Rideal \rho$ belongs to $\ord(\mu)$, the relation (\ref{eq:rhomu}) must hold for some $\mu'$ of the same norm as $\mu$. If we now multiply $b + \omega = \xi\rho$ by $\rho$ from the left and from the right, we shall get $b + \omega' = \rho\xi$, where $\omega' = (r - 1 + \mu')/r$. Therefore, $2b + \omega + \omega' = \rho\xi + \xi\rho$, and from the relation (\ref{eq:qrrq}) we obtain the following system of equations:

\begin{equation} \label{eq:system}
\left\{
\begin{array}{l}
b = \Re(\rho)\Re(\xi) - (\rho,\xi) - \frac{r - 1}{r};\\
\frac{\omega + \omega'}{2} = \Re(\xi)\vec \rho + \Re(\rho)\vec \xi + \frac{r - 1}{r}.
\end{array}
\right.
\end{equation}

From the proposition \ref{prop:equal_scalar_products_rho} it follows that $(\omega,\rho) = (\omega',\rho)$, and therefore $\left(\frac{\omega + \omega'}{2}, \rho\right) = (\omega, \rho)$. If we now combine this statement with the second equality of the system (\ref{eq:system}), we shall obtain

$$
(\omega, \rho) = \left(\frac{\omega + \omega'}{2}, \rho\right) = \left(\Re(\rho)\vec \xi + \Re(\xi)\vec \rho + \frac{r - 1}{r}, \rho\right) =
$$

$$
= \Re(\rho)(\rho,\xi) + \Re(\xi)(\rho,\rho) = \Re(\rho)(\rho,\xi) + \Re(\xi)\left(N(\rho) - \Re(\rho)^2\right).
$$

From the first equality of the system (\ref{eq:system}), we know that $(\rho,\xi) = \Re(\rho)\Re(\xi) - b - (r - 1) / r$. Now, a simple substitution of $(\rho, \xi)$ in the equation above will give us the relation (\ref{eq:main_relation_rho}):

$$
(\omega,\rho) = \Re(\rho)\left(\Re(\rho)\Re(\xi) - b - \frac{r - 1}{r}\right) + \Re(\xi)\left(N(\rho) - \Re(\rho)^2\right) = 
$$

$$
= -\Re(\rho)\left(b + \frac{r - 1}{r}\right) + \Re(\xi)N(\rho).
$$
\end{proof}

This proposition allows us to formulate the following algorithm, which computes the $\RZ$-basis of an ideal $\Rideal \rho$ in the quadratic order $\ord(\mu)$.

\begin{alg} \label{alg:restore}
\emph{Restoring the $\RZ$-basis of an ideal by its pseudo generator.}

\par \textbf{Input:} a quadratic order $\ord(\mu)$ and a primitive ideal $\Rideal \rho$, which belongs to this order;
\par \textbf{Output:} Integers $a$ and $b$ s.t. $\Rideal \rho = [a, b + \omega]$.

\begin{enumerate}[1.]
\item Find a unit $\varepsilon$ s.t. $\Re(\varepsilon\rho)$ is an odd integer. Put $\rho := \varepsilon\rho$;

\item If $N(\rho)$ is even and $N(\mu)$, $(\mu,\rho)$ have different parity, set $t := 2$. Otherwise set $t := 1$;

\item Find $d = \gcd\left(\Re(\rho), N(\rho) / t\right)$;

\item Using the extended Euclidean algorithm, compute the smallest non-negative $b$, s.t. $N(\rho) \leq N(b + \omega)$, and

$$
rb + r - 1 \equiv - \frac{(\mu,\rho)}{d}\left(\frac{\Re(\rho)}{d}\right)^{-1} \,\,\, \left(\textnormal{mod} \,\,\, \frac{N(\rho)}{td}\right);
$$

\item Return $a = N(\rho)$ and $b$.
\end{enumerate}
\end{alg}

\begin{proof_cor}
The work of this algorithm is based on the relation (\ref{eq:main_relation_rho}). Obviously, if $\Re(\rho) = 0$, then the summand which contains $b$, vanishes. This is why we redefine $\rho$ so that its real part would not be equal to zero. In this case, we pick $\varepsilon$ to be equal to either $i$, $j$, or $k$, depending on which coefficient of $\rho$ is odd.

\par We would also like to avoid the case when $\rho$ has half-integer coefficients. Fortunately, there will always exist exactly \emph{eight} units of the form $(\pm 1 \pm i \pm j \pm k)/2$, which will make $\varepsilon\rho$ have integral coefficients. In particular, if $\rho$ has even (odd) number of coefficients congruent to 1 mod 4, then all units of the type mentioned above, which have an odd (even) number of 1's, will work.

\par Finally, we argue that a pseudo generator $\rho$ with integral coefficients will always have at least one odd coefficient, for if this is not the case, then all coefficients of $\rho$ are even, which makes $\rho$ non-primitive.

\par Now that our $\rho$ has an odd real part, we want to make sure that all the other parts of the relation (\ref{eq:main_relation_rho}) are not half-integers. If we multiply this equality by $r$, and keep in mind that $r(\omega,\rho) = (\mu,\rho)$, we get

$$
(\mu, \rho) + \Re(\rho)(rb + r - 1) = r\Re(\xi)N(\rho).
$$

If $r = 2$, then all the numbers above are integers, including $r\Re(\xi)$. However, if $r = 1$,  then we have to check whether $\Re(\xi)$ is a half-integer. This can happen only if $N(\rho)$ is even, because the value on the left side can never be a half-integer.

\par Suppose that $N(\rho)$ is even. Let us demonstrate that $N(\rho) \equiv 2$ (mod 4). Recall that $N(\rho)$ divides $N(b + \omega) = b^2 + m$, where $m = N(\mu)$. If $m$ is odd, then $b$ must be odd as well, since 2 divides $b^2 + m$. Because $m \equiv 1$ (mod 4) and $b^2 \equiv 1$ (mod 4), we have $b^2 + m \equiv 2$ (mod 4), so $N(\rho)$ cannot be divisible by 4. On the other hand, if $m$ is even, then $m \equiv 2$ (mod 4), because $m$ is squarefree. For the reason that $m$ and $b$ have to be of the same parity, the congruence $b^2 \equiv 0$ (mod 4) holds, and therefore $b^2 + m \equiv 2$ (mod 4). Once again, this gives us the desirable result.

\par Now we know: if $N(\rho)$ is even and $\xi$ has half-integer coefficients, then the value on the right side of the previous equation has to be odd. Therefore, the value on the left side must be odd as well. When $m$ (and therefore $b$) is odd, this can happen only if $(\mu,\rho)$ is even, because $\Re(\rho)b$ will always be odd. On the other hand, if $m$ (and therefore $b$) is even, this happens only when $(\mu,\rho)$ is odd. Summarizing, when $N(\rho)$ is even and $\Re(\rho)$ is odd, then we can say that $\xi$ has half-integer coefficients if and only if $m$ and $(\mu,\rho)$ have different parity.

\par If we now set $t$ as it is described at the step 2, then the following equation has no half-integers involved, independently of the input to the algorithm:

$$
(\mu,\rho) + \Re(\rho)(rb + r - 1) = \underbrace{tr\Re(\xi)}_{\textnormal{integer}}\frac{N(\rho)}{t}.
$$

\par Define $d = \gcd\left(\Re(\rho), N(\rho) / t\right)$. Because $d$ divides both $\Re(\rho)$ and $N(\rho)/t$, it has to divide $(\mu,\rho)$ as well. If we now divide both sides of the equation above by $d$, then $\Re(\rho) / d$ and $N(\rho)/(td)$ have no common divisor greater than $1$, and therefore we can apply the extended Euclidean algorithm to find the value of $rb + r - 1$, using the formula introduced in step 4 of the algorithm. Finally, when computing $b$, we would also like to make sure that $N(\rho) \leq N(b + \omega)$, just in case if $N(\rho) > m$.
\end{proof_cor}

\begin{rem}
Analogous algorithm exists for the ring of polynomials \cite[Alg. IDEAL]{scheidler}.
\end{rem}

\begin{rem}
Algorithm \ref{alg:restore} resides on the Euclidean algorithm, so its time complexity (in the worst case) is estimated by $O\left(\log^2\Re(\rho)\right)$, where the right pseudo generator $\rho$ is chosen so that $\Re(\rho)$ is positive and odd.
\end{rem}

\subsection{Computing the right pseudo generator of an ideal using its left pseudo generator}

The next problem of a great interest is the following: given an ideal $\Lideal \rho$ in the quadratic order $\ord(\mu)$, find an integral quaternion $\rho'$ s.t. $\Lideal \rho = \Rideal{\rho'}$.

\par Consider the $\RZ$-basis of an ideal: $\Lideal \rho = \Rideal{\rho'} = [a, b + \omega]$. Recall that $\rho' = \gcd_r(a, b + \omega)$, so the most straightforward way to compute $\rho'$ is to restore the $\RZ$-basis using the algorithm \ref{alg:restore}, and then apply the Euclidean algorithm to $a$ and $b + \omega$. However, as was mentioned earlier, in this paper we try to avoid the usage of the $\RZ$-basis in our computations. In this section, we present another approach to find $\rho'$, without obtaining $b$ in the first place.

\par Recall that for any integers $X$ and $Y$, a quaternion $aX + (b + \omega)Y$ belongs to the ideal, and therefore has to be divisible by $\rho'$ from the right. Since the relation (\ref{eq:main_relation_rho}) holds for the left pseudo generator as well, if we now set $X = -\Re(\xi)$ and $Y = \Re(\rho)r$, we shall obtain

$$
aX + (b + \omega)Y = (aX + bY) + \omega Y =
$$

$$
 = \left(-\Re(\xi)N(\rho)r + \Re(\rho)br\right) + r\Re(\rho)\omega = -r(\omega,\rho) - \Re(\rho)(r - 1) + r\Re(\rho)\omega.
$$

Recall that $a = N(\rho)$ and $\xi$ satisfies the relation $b + \omega = \rho\xi$. Since $r(\omega,\rho) = (\mu,\rho)$ and $-\Re(\rho)(r - 1) + r\Re(\rho)\omega = \Re(\rho)\mu$, we conclude that

$$
aX + (b + \omega)Y = -r(\omega,\rho) - \Re(\rho)(r - 1) + r\Re(\rho)\omega = -(\mu,\rho) + \Re(\rho)\mu.
$$

Now an application of the Euclidean algorithm will help us to derive $\rho'$:

\begin{equation} \label{eq:right_by_left_1}
\rho' = \gcd\,\!_r\left(-(\mu,\rho) + \Re(\rho)\mu, N(\rho)\right).
\end{equation}

\par Consider another approach to the same problem. Since $\rho$ is the left pseudo generator of an ideal which belongs to $\ord(\mu)$, then there exists an integral quaternion $\mu'$ of the same norm as $\mu$, s.t. $\rho\mu' = \mu\rho$. If we now add $\rho\mu$ on both sides of this equation, and then apply the relation (\ref{eq:qrrq}), we shall get the following result:

$$
\rho\mu + \rho\mu' = \rho\mu + \mu\rho;
$$

$$
\rho\frac{\mu + \mu'}{2} = \frac{\mu\rho + \rho\mu}{2} = -(\mu,\rho) + \Re(\rho)\mu.
$$

So, the following formula can also be used to compute $\rho'$:

\begin{equation} \label{eq:right_by_left_2}
\rho' = \gcd\,\!_r\left(\rho\frac{\mu + \mu'}{2}, N(\rho)\right).
\end{equation}

\begin{rem}
Analogously, we can solve the problem of finding the left pseudo generator of an ideal given its right pseudo generator. In that case, the formula (\ref{eq:right_by_left_1}) does not change, except for $\gcd_r$, which gets replaced by $\gcd_l$.  In turn, the formula (\ref{eq:right_by_left_2}) takes the form

$$
\rho' = \gcd\,\!_l\left(\frac{\mu+\mu'}{2}\rho, N(\rho)\right).
$$
\end{rem}

\subsection{Ideal multiplication}

Let us now turn our attention to the problem of ideal multiplication. Consider a quadratic order $\ord(\mu)$ and two ideals, $\Rideal \rho$ and $\Rideal{\rho'}$, which belong to this quadratic order. The following algorithm multiplies these ideals using solely their pseudo generators.

\begin{alg} \emph{Ideal multiplication.}

\par \textbf{Input:} a quadratic order $\ord(\mu)$, and two primitive ideals, $\Rideal \rho$ and $\Rideal{\rho'}$, which belong to this quadratic order;
\par \textbf{Output:} (non-reduced) ideal $\Rideal \rho \cdot \Rideal{\rho'}$.

\begin{enumerate}[1.]
\item Compute $\mu' = \rho'\mu\rho'^{-1}$;
\item Compute $\rho'' = \gcd_r\left(-(\mu,\rho) + \Re(\rho)\mu', N(\rho)\right)$;
\item Return $\Rideal{\rho''\rho'}$.
\end{enumerate}
\end{alg}

\begin{proof_cor}
\par The problem of ideal multiplication was studied by Venkov in his \nth{4} letter \cite[\textsection 18]{ven4}. Let $\mathfrak a = \Rideal \rho$ and $\mathfrak b = \Rideal{\rho'}$. Venkov demonstrated that if an ideal $\mathfrak a = [a, b + *]$ (where $*$ can be replaced by $\omega$, $\omega'$, etc. to move an ideal between various quadratic orders) has a right pseudo generator $\rho''$ in a quadratic order $\ord(\mu')$, where $\mu' = \rho'\mu\rho'^{-1}$, then $\rho''\rho'$ is a right pseudo generator of an ideal $\mathfrak{ab}$ in $\ord(\mu)$. The question is, how do we find $\rho''$?

\begin{displaymath}
    \xymatrix{
\mu \ar[d]_{\rho} \ar@/_/@{.>}[dr]^{\rho''\rho'} \ar[r]^{\rho'} & \mu' \ar[d]^{\rho''}\\
\mu''' & \mu''
}
\end{displaymath}

\par We can compute $\rho''$ using an approach similar to the one described in the previous subsection. Let $\Rideal \rho = [a, b + \omega]$. Then $[a, b + \omega']$, where $\omega' = \rho'\omega\rho'^{-1}$, belongs to the quadratic order $\ord(\mu')$. For any integers $X$ and $Y$, the quaternion $aX + (b + \omega')Y$ belongs to this ideal, and therefore has to be divisible by $\rho''$ from the right. Since the relation (\ref{eq:main_relation_rho}) holds for the ideal $[a, b + \omega]$ in the quadratic order $\ord(\mu)$, then we can set $X = -\Re(\xi)$ and $Y = \Re(\rho)r$, which will give us

$$
-r(\omega,\rho) - \Re(\rho)(r - 1) + r\Re(\rho)\omega' = -(\mu,\rho) + \Re(\rho)\mu'.
$$

This integral quaternion belongs to $[a, b + \omega']$. Now $\rho''$ can be obtained using the Euclidean algorithm, as it is described in the step 2.

\par Now that we went over the steps of the algorithm, let us prove that $\rho''\rho'$ is really a pseudo generator of the product. Consider the product of $\Rideal \rho = [a, b + \omega]$ and $\Rideal{\rho'} = [a', b' + \omega]$:

$$
[a, b + \omega] \cdot [a', b' + \omega] = \left[aa', a(b' + \omega), a'(b + \omega), (b' + \omega)(b + \omega)\right].
$$

Let $b' + \omega = \xi'\rho'$ and $b + \omega' = \xi''\rho''$. Then

$$
aa' = \overline{\rho'}\,\,\overline{\rho''}\rho''\rho';
$$

$$
a(b' + \omega) = a\xi'\rho' = \xi'a\rho' = \xi'\overline{\rho''}\rho''\rho';
$$

$$
a'(b + \omega) = \overline{\rho'}\underbrace{\rho'(b + \omega)\rho'^{-1}}_{b + \omega'}\rho' = \overline{\rho'}\xi''\rho''\rho';
$$

$$
(b' + \omega)(b + \omega) = \xi'\underbrace{\rho'(b + \omega)\rho'^{-1}}_{b + \omega'}\rho' = \xi'\xi''\rho''\rho'.
$$

Since $b + \omega' = \xi''\rho''$ is primitive, quaternions $\xi''$ and $\overline{\rho''}$ do not have a common divisor from the right, and therefore $\rho''\rho'$ is the greatest common divisor of $aa'$, $a(b' + \omega)$, $a'(b + \omega)$ and $(b'+\omega)(b + \omega)$. Since these four quaternions form a $\RZ$-basis of a resulting ideal, $\rho''\rho'$ has to be the right pseudo generator of this ideal.
\end{proof_cor}

\begin{rem}
Our proof is similar to the proof for the multiplication of ideals in the ring of ad\`eles \cite{kubensky}.
\end{rem}

\subsection{Ideal reduction}

Before we turn our attention to the process of ideal reduction, let us first state some important theorems.

\begin{thm} \label{thm:ideal_generated_by_orders}
Let $\ord(\mu)$ be a quadratic order, $m = N(\mu)$. If a primitive integral quaternion $\rho$ satisfies the relation (\ref{eq:rhomu}) for some $\mu'$, and $N(\rho)$ is odd when $m \equiv 3 \,\, \textnormal{(mod 8)}$, then it is a right (left) pseudo generator of some ideal in the quadratic order $\ord(\mu)$ ($\ord(\mu')$). Conversely, if $\Rideal \rho$ is an ideal in the quadratic order $\ord(\mu)$, then $\rho$ satisfies (\ref{eq:rhomu}) for some $\mu'$, and $N(\rho)$ is odd when $m \equiv 3 \,\, \textnormal{(mod 8)}$. Moreover, $\Lideal \rho$ is an ideal in $\ord(\mu')$.
\end{thm}

\begin{proof}
\par Let us first argue why it is necessary for $N(\rho)$ to be odd when $m \equiv 3$ (mod 8). If we consider a primitive ideal $[N(\rho), b + \omega]$, then  $N(\rho) \,\,| \,\, N(b + \omega) = \frac{(2b + 1)^2 + m}{4}$, so $-m$ has to be a quadratic residue of $4N(\rho)$. Since any odd integer $x$ satisfies the congruence $x^2 \equiv 1$ (mod 8), we have $(2b + 1)^2 + m \equiv 4$ (mod 8), and therefore $N(\rho)$ has to be odd. In his \nth{2} letter, Venkov proved that for any $\rho$ and $\mu$, which satisfy the relation (\ref{eq:rhomu}) for some $\mu'$, $-m$ has to be a quadratic residue of $N(\rho)$  \cite[\textsection 10]{ven2}. However, when $m \equiv 3$ (mod 8), this does not imply that $-m$ is a quadratic residue mod $4N(\rho)$. In fact, the implication takes place only if $N(\rho)$ is odd. In that case, it is easy to see that if $x^2 \equiv -m$ (mod $N(\rho)$) for some $x$, then $4N(\rho)$ divides $x^2 + m$ if $x$ is odd, and $(x + N(\rho))^2 + m$ if $x$ is even. Note that if $N(\rho)$ is even while $m \equiv 3$ (mod 8), then an ideal $\Rideal \rho$ would belong to $[1, \mu]$, which is not isomorphic to the maximal quadratic order $\ord = \left[1, (1 + \sqrt{-m})/2\right]$ of $\QF{-m}$.

\par ($\Rightarrow$) Let $\rho$ satisfy the equation $\rho\mu = \mu'\rho$ for some $\mu'$, and $N(\rho)$ be odd when $m \equiv 3$ (mod 8). Then $-m$ is a quadratic residue of $r^2N(\rho)$, where $r$ is defined as in (\ref{eq:omega}). Define $b$ as a square root of $-m$ mod $r^2N(\rho)$. Then $[N(\rho), b + \omega]$ is an ideal in $\ord(\mu)$. Let $b + \omega = \xi\rho$ for some $\xi$. Since $b + \omega$ is primitive, quaternions $\xi$ and $\overline{\rho}$ have no common divisor from the right, which makes $\rho$ a greatest common divisor of $N(\rho)$ and $b + \omega$ from the right, and therefore a right pseudo generator of $[N(\rho), b + \omega]$. Since $N(\mu) = N(\mu')$, a similar argument can be applied towards the ideal $[N(\rho), b + \omega']$, where $b + \omega' = \rho(b + \omega)\rho^{-1} = \rho\xi$.

\par ($\Leftarrow$) The proof of the converse statement was demonstrated in the section \ref{sec:background}.
\end{proof}

Next, we must point out that some quadratic orders are equivalent to each other. Note that for any unit $\varepsilon$ we have $\Rideal{\varepsilon\rho} = \Rideal \rho$. Hence, if $\rho$ satisfies the relation $\rho\mu = \mu'\rho$, then $\varepsilon\rho$ satisfies

$$
\varepsilon \rho \cdot \mu = \varepsilon\mu'\overline{\varepsilon}\cdot\varepsilon\rho.
$$

It means that our $\mu'$ in the relation (\ref{eq:rhomu}) may vary, depending on the pseudo generator we choose. As a result, we can define the following equivalence relation on the set of all quadratic orders in $\iquat$ of a certain discriminant.

\begin{defin}
We say that two quadratic orders $\ord(\mu)$ and $\ord(\mu')$ are \emph{equivalent}, if $\mu$ and $\mu'$ satisfy the relation $\mu' = \varepsilon\mu\overline{\varepsilon}$ for some unit $\varepsilon$.
\end{defin}

It is well known that ideals, which are equivalent to each other, form an \emph{equivalence class}. The following theorem describes how we can fully characterize this set.

\begin{thm} \label{thm:equivalence_class}
\emph{\cite[\textsection 9]{ven3}} Consider a quadratic order $\ord(\mu)$, which contains two equivalent ideals $\Rideal \rho$ and $\Rideal{\rho'}$. If $\rho\mu = \mu'\rho$ and $\rho'\mu=\mu''\rho'$, then $\ord(\mu')$ and $\ord(\mu'')$ are equivalent.
\end{thm}

The theorem above demonstrates a rather interesting fact that all solutions to the equation $\rho\mu= \mu'\rho$ for given $\mu$ and $\mu'$ are pseudo generators of ideals, which are equivalent to each other.

\par In his \nth{3} letter, Venkov described the full set of solutions to the equation (\ref{eq:rhomu}). We summarize his results in the following theorem.

\begin{thm} \label{thm:solutions} \emph{\cite{ven3, ven4}}
Let $\mu$ and $\mu'$ be two primitive integral quaternions with the zero real part, s.t. $N(\mu) = N(\mu')$. Then the set of all solutions to the equation $\rho\mu = \mu'\rho$ takes the form of a $\RZ$-module $[\upsilon, \upsilon_1]$ for some particular integral quaternions $\upsilon$, $\upsilon_1$.
\end{thm}

\begin{proof}
In this proof, we will simply describe the form of quaternions $\upsilon$, $\upsilon_1$. For more details, we ask our reader to refer to the original works of Venkov. Let

$$
\begin{array}{l}
\mu = xi + yj + zk,\\
\mu' = x'i + y'j + z'k.
\end{array}
$$

If $\mu' \neq -\mu$, then the set of all $\rho$ with \emph{integral} coefficients is of the form \cite[\textsection 11]{ven3}

$$
\rho = \upsilon X + \Upsilon Y,
$$

where $X, Y$ are integers, and $\upsilon$, $\Upsilon$ are two integral quaternions, defined by the following system of equations:

\begin{equation} \label{eq:integral_solutions}
\left \{
\begin{array}{l}
\upsilon = \frac{x' + x}{d}i + \frac{y' + y}{d} j + \frac{z' + z}{d}k = \frac{\mu'+\mu}{d},\\

\Upsilon = -\frac{d}{e} + \frac{p}{e}i + \frac{q}{e}j + \frac{r}{e}k.
\end{array}
\right.
\end{equation}

In this system, $d$ and $e$ are defined as follows:

$$
d = \gcd\left(x' + x, y' + y, z' + z\right),
$$

$$
e = \gcd\left(x' + x, y' + y, z' + z, x' - x, y' - y, z' - z\right).
$$

Integers $a$, $b$ and $c$ are defined using the extended Euclidean algorithm:

$$
a(x' + x) + b(y' + y) + c(z' + z) = d.
$$

Integers $p$, $q$ and $r$ are defined through $a$, $b$, $c$:

$$
p = c(y' - y) - b(z' - z), \,\,\,\,\, q = a(z' - z) - c(x' - x), \,\,\,\,\, r = b(x' - x) - a(y' - y).
$$

The full set of solutions (including those $\rho$ with half-integer coefficients) has the form \cite[\textsection 12]{ven3}

\begin{equation} \label{eq:all_solutions}
\rho = \upsilon X + \upsilon_1 Y,
\end{equation}

where $\upsilon$ is defined as in system (\ref{eq:integral_solutions}), and $\upsilon_1$ is defined as in one of the following four cases:

\begin{enumerate}[(a)]
\item
{
If $m \equiv 1, 2$ (mod 4), and one of the sums $x' + x$, $y' + y$, $z' + z$, for example $x' + x$, is even, then

$$
\upsilon_1 = \frac{1}{2}\left(\Upsilon + (a + 1) \upsilon\right);
$$
}

\item
{
If $m \equiv 3$ (mod 8) and one of the sums, for example $x' + x \equiv 2$ (mod 4), and others are $\equiv 0$ (mod 4), then

$$
\upsilon_1 = \frac{1}{2}\left(\Upsilon + (b + c + 1)\upsilon\right),
$$
}

\item
{
If $m \equiv 3$ (mod 8) and all of the sums are $\equiv 2$ (mod 4), then

$$
\upsilon_1 = \frac{1}{2}\left(\Upsilon + \upsilon\right);
$$
}

\item
{
Otherwise $\upsilon_1 = \Upsilon$.
}
\end{enumerate}

In case if $\mu' = -\mu$, define $d = \gcd(y, z)$. If $e$, $f$ are solutions to the equation

$$
ey + fz = d,
$$

then the set of all $\rho$ satisfying the equation $\rho\mu = -\mu\rho$ is of the form (\ref{eq:all_solutions}), where

$$
\upsilon = \frac{zj - yk}{d}, \,\,\, \upsilon_1 = -di + x(ej + fk).
$$
\end{proof}

While the theorem \ref{thm:equivalence_class} tells us that the relation (\ref{eq:rhomu}) fully characterizes some particular equivalence class, the theorem \ref{thm:solutions} tells us that all pseudo generators of ideals in this equivalence class have to be of the form $\rho = \upsilon X + \upsilon_1 Y$ for some integers $X, Y$ and some fixed integral quaternions $\upsilon$, $\upsilon_1$.

\hspace{2cm}

\par Now we have all the results we need to turn our attention to the problem of ideal reduction.

\begin{defin} \label{def:reduced}
An ideal $\Rideal r$ in the quadratic order $\ord(\mu)$ is said to be \emph{reduced}, if in the set of solutions to the equation (\ref{eq:rhomu}), where $\mu' = r\mu r^{-1}$, it has the smallest positive norm.
\end{defin}

By the theorem \ref{thm:solutions}, the set of all solutions to the equation (\ref{eq:rhomu}) is equal to the $\RZ$-module $[\upsilon, \upsilon_1]$. In turn, the theorem \ref{thm:equivalence_class} tells us that any ideal $\Rideal{\upsilon X + \upsilon_1 Y}$ is equivalent to $\Rideal r$. Therefore, our goal is to find an integral quaternion of the smallest positive norm, which belongs to $[\upsilon, \upsilon_1]$.

\par It is quite straightforward to find the norm of the reduced ideal. Consider the norm of $\rho = \upsilon X + \upsilon_1 Y$:

\begin{equation} \label{eq:norm}
N(\rho)(X,Y) = N(\upsilon)X^2 + 2(\upsilon,\upsilon_1)_\times XY + N(\upsilon_1)Y^2.
\end{equation}

Since $N(\rho)(X,Y)$ is a quadratic form, then the problem of finding the smallest norm is equivalent to the problem of finding the smallest integer, representable by the form $\left(N(\upsilon), 2(\upsilon,\upsilon_1)_\times, N(\upsilon_1)\right)$. Using the reduction algorithm for quadratic forms \cite[Alg. 5.6.2]{crandall}, we can find the reduced form $(a, b, c)$, equivalent to the given one. The form $(a, b, c)$ has a property that $a$ is the smallest integer representable by this form, and therefore it is also the smallest integer representable by the original form \cite[Sec. 5.10]{buchmann}.

\par As we have seen, given an ideal $\Rideal \rho$ in the quadratic order $\ord(\mu)$, it is possible to find the norm of the reduced ideal $\Rideal r$, but not the ideal itself. The following algorithm allows us to compute a pseudo generator of the reduced ideal. Note that $\lfloor x \rceil = \lfloor x + 1/2\rfloor$ defines an integer, which is closest to $x$.

\begin{alg} \label{alg:reduction} \emph{Ideal reduction.}

\par \textbf{Input:} a quadratic order $\ord(\mu)$ and a primitive ideal $\Rideal \rho$, which belongs to this order;
\par \textbf{Output:} a reduced ideal, which is equivalent to $\Rideal \rho$.

\begin{enumerate}[1.]
\item Compute $\mu' := \rho\mu\rho^{-1}$;

\item Compute the $\RZ$-basis of the module $[\upsilon, \upsilon_1]$, $N(\upsilon) > N(\upsilon_1)$, which characterizes the set of solutions to the equation $\rho\mu = \mu'\rho$;

\item Switch $\upsilon$ and $\upsilon_1$;

\item Compute $X := \lfloor (\upsilon, \upsilon_1)_\times / N(\upsilon) \rceil$;

\item Put $\omega_1 := \omega_1 - \omega X$;

\item If $N(\upsilon) > N(\upsilon_1)$, go to step 3;

\item If $N(\upsilon) = N(\upsilon_1)$ and $2\left|(\upsilon, \upsilon_1)_\times\right| \geq N(\upsilon)$, return $\Rideal{\upsilon_1 - \upsilon}$. Otherwise return $\Rideal{\upsilon}$.
\end{enumerate}
\end{alg}

\begin{proof_cor}
\par Let $N(\upsilon_1) \geq N(\upsilon)$. Consider the norm of a quaternion $\upsilon_1 - \upsilon X$:

\begin{equation} \label{eq:norm_quadratic}
N(\upsilon_1 - \upsilon X) = N(\upsilon)X^2 - 2(\upsilon,\upsilon_1)_\times X + N(\upsilon_1).
\end{equation}

This quadratic function of one variable reaches its minimum at the point $X_0 = (\upsilon, \upsilon_1)_\times / N(\upsilon)$. However, we are interested in integral (or half-integral) $X$, s.t. $N(\upsilon_1 - \upsilon X)$ takes its smallest positive value, and $\upsilon_1 - \upsilon X$ is an integral quaternion. Note that $X$ can be a half-integer only in one of the following two cases:

\begin{enumerate}[(a)]
\item When coefficients of $\upsilon_1$ are integers, and coefficients of $\upsilon$ are integral and odd;
\item When coefficients of $\upsilon_1$ are half-integers, and coefficients of $\upsilon$ are integral and even.
\end{enumerate}

\par Let us demonstrate that during the work of algorithm \ref{alg:reduction} quaternions $\upsilon$, $\upsilon_1$ cannot satisfy neither (a), nor (b). Note that if any of the above conditions hold, then $\upsilon$ is divisible by 2, and therefore it is non-primitive. We aim to show that $\upsilon$ and $\upsilon_1$ will always be primitive.

\par The first time we reach the step 3 of our algorithm, $\upsilon$ and $\upsilon_1$ are primitive and linearly independent, because of the way they were defined in the theorem \ref{thm:solutions}. On the step 5, we define $\upsilon_1 := \upsilon_1 - \upsilon X$, so now our new $\upsilon_1$ may not be primitive. However, we can demonstrate that this is not the case.

\par Suppose that in the $\RZ$-module $[\upsilon, \upsilon_1]$ we have $\upsilon_1 = \Delta q$ for some integer $\Delta > 1$ and a primitive quaternion $q$. Since $\upsilon_1$ is a solution to the equation (\ref{eq:rhomu}), then $q$ is a solution as well. It means that $q = \upsilon X + q \Delta Y$ for some integers $X$ and $Y$, because $[\upsilon, \upsilon_1]$ completely characterizes the set of solutions of the equation (\ref{eq:rhomu}). Therefore, both of the following equalities

$$
\upsilon = \frac{1 - \Delta Y}{X} q, \,\,\, q = \frac{X}{1 - \Delta Y} \upsilon
$$

must hold at the same time. Note that $X \neq 0$, since $\Delta > 1$. Now, because $q$ is primitive, $X$ must divide $1 - \Delta Y$, and since $\upsilon$ is primitive, $1 - \Delta Y$ must divide $X$, whence $X = 1 - \Delta Y$. It implies that $\upsilon$ is equal to $q$, which contradicts the fact that $\upsilon$ and $\upsilon_1$ are linearly independent. Therefore, $\upsilon$ and $\upsilon_1$ are both primitive, which means that $X$ cannot be a half-integer.

Let us now return back to our algorithm. As it was mentioned before, the function (\ref{eq:norm_quadratic}) reaches its minimum at the point $X_0 = (\upsilon, \upsilon_1)_\times / N(\upsilon)$. Therefore, an integer $X = \lfloor X_0 \rceil$ corresponds to an integral quaternion $\upsilon_1 - \upsilon X$, which has the smallest norm. Now we can replace the $\RZ$-module $[\upsilon, \upsilon_1]$ with the equivalent $\RZ$-module $[\upsilon, \upsilon_1 - \upsilon X]$. However, if $N(\upsilon_1) > N(\upsilon)$, where $\upsilon_1 := \upsilon_1 - \upsilon X$, there is no need to go back to step 3. For if we would perform the norm reduction in the same way one more time (without swapping $\upsilon$ and $\upsilon_1$, because $N(\upsilon_1)$ is already greater than $N(\upsilon)$), we would get $X = 0$. As we shall see later, $X = 0$ actually indicates that $\Rideal \upsilon$ is reduced. In case if $N(\upsilon) = N(\upsilon_1)$, we also don't need to go back to step 3. If we consider the quadratic form (\ref{eq:norm}), which characterizes the norm of $\rho = \upsilon X + \upsilon_1 Y$, the Cauchy-Schwarz inequality holds:

$$
\left|(\upsilon, \upsilon_1)_\times\right| \leq \sqrt{N(\upsilon)\cdot N(\upsilon_1)} = N(\upsilon).
$$

It means that $X$ is either $0$ (i.e. $\Rideal \upsilon$ is reduced), or $\pm 1$, depending on the sign of $(\upsilon,\upsilon_1)_\times$. In the second case, we replace a $\RZ$-module $[\upsilon, \upsilon_1]$ with $[\upsilon_1 - \upsilon, \upsilon]$. Note that $N(\upsilon_1 - \upsilon) \leq N(\upsilon)$ because $2\left|(\upsilon,\upsilon_1)_\times\right| \geq N(\upsilon)$ (i.e. $|X| = 1$). Now, compute

$$
X = \left\lfloor \frac{(\upsilon_1 - \upsilon, \upsilon)_\times}{N(\upsilon_1 - \upsilon)} \right\rceil = \left\lfloor \frac{(\upsilon, \upsilon_1)_\times - N(\upsilon)}{N(\upsilon_1) - 2(\upsilon, \upsilon_1)_\times + N(\upsilon)} \right\rceil = \left\lfloor -\frac{1}{2} \right\rceil = 0.
$$

Because $X = 0$, we conclude that $\Rideal{\upsilon_1 - \upsilon}$ is reduced.

\par Summing up, we have to replace $[\upsilon, \upsilon_1]$ with $[\upsilon, \upsilon_1 - \upsilon X]$ only if $N(\upsilon) > N(\upsilon_1)$. When this is done, at the step 3 we swap the basis quaternions, in order to ensure that at the step 4 we have $N(\upsilon_1) > N(\upsilon)$. Note that on each iteration we ensure that the norm of one of the basis quaternions gets lessened. Since norms are bounded below by zero, and we return the result whenever $N(\upsilon_1) \geq N(\upsilon)$, our algorithm will definitely terminate.

\par Now, let us prove that if $X = 0$, then the norm of $\upsilon$ is the smallest positive integer, which can be represented by a quadratic form $\left(N(\upsilon), 2(\upsilon,\upsilon_1)_\times, N(\upsilon_1)\right)$. First of all, recall that the condition $X = \lfloor (\upsilon, \upsilon_1)_\times / N(\upsilon) \rceil = 0$ means

\begin{equation} \label{eq:X=0}
-\frac{N(\upsilon)}{2} \leq (\upsilon, \upsilon_1)_\times < \frac{N(\upsilon)}{2}.
\end{equation}

Second, suppose that there exists a quaternion $r$ s.t. $N(r) \leq N(\upsilon)$. Then there exist two integers $X$ and $Y$ s.t. $r = \upsilon X + \upsilon_1 Y$, i.e.

$$
N(r) = N(\upsilon)X^2 + 2(\upsilon, \upsilon_1)_\times XY + N(\upsilon_1)Y^2 \leq N(\upsilon).
$$

Since $(\upsilon, \upsilon_1)_\times$ satisfies the relation (\ref{eq:X=0}),  we have

$$
N(\upsilon) \geq N(r) \geq N(\upsilon)X^2 - N(\upsilon)|XY| + N(\upsilon_1)Y^2. 
$$

On the \nth{3} step, we ensure that $\upsilon_1$ has a greater norm than $\upsilon$, so the following inequality holds:

$$
N(\upsilon) \geq N(r) \geq N(\upsilon) X^2 - N(\upsilon)|XY| + N(\upsilon)Y^2 \geq 0.
$$

If we now divide all the parts of the inequality above by $N(\upsilon) \neq 0$, we get

$$
1 \geq \frac{N(r)}{N(\upsilon)} \geq X^2 - |XY| + Y^2 \geq 0.
$$

Since $X^2 - |XY| + Y^2$ takes only integral values, it has to be equal to either 0 or 1. If it is equal to zero, then $X = 0$ and $Y = 0$, because $X^2 - |XY| + Y^2$ is a positive definite quadratic form \cite[Prop. 1.2.10]{buchmann}. This implies $N(r) = 0$. On the other hand, if $X^2 - |XY| + Y^2 = 1$, then $N(r) = N(\upsilon)$. Therefore, $\upsilon$ has the smallest positive norm.
\end{proof_cor}

The following theorem describes an upper bound to the number of iterations that algorithm \ref{alg:reduction} performs in order to compute the reduced ideal.

\begin{thm} \label{thm:reduce_bound}
Let $[\alpha_0, \beta_0]$ be a $\RZ$-module, where $\alpha_0$ and $\beta_0$ are two arbitrary linearly independent integral quaternions with $A = N(\alpha_0) \geq B = N(\beta_0)$. Define an algorithm, which computes

$$
\alpha_{n + 1} = \beta_n, \,\,\, \beta_{n+1} = \alpha_n - \left\lfloor\frac{(\alpha_n, \beta_n)_\times}{N(\beta_n)}\right\rceil \beta_n,
$$

and terminates whenever $N(\beta_k) \geq N(\alpha_k)$ for some $k$. Then $k(A,B) \in O\left(\log\left|\log\frac{A}{B\varphi}\right|\right)$ as $A/B \rightarrow \varphi$, where $\varphi = (1 + \sqrt 5) / 2$. In case if $B = \lfloor A / \varphi \rfloor$, then $k(A) \in O\left(\log \frac{A}{\{A / \varphi\}}\right)$ as $A \rightarrow \infty$, where $\{x\}$ denotes the fractional part of $x$.
\end{thm}

\begin{proof}
Let us estimate the difference between the norms of quaternions $\alpha_n$ and $\beta_{n+1}$:

$$
N(\alpha_n) - N(\beta_{n+1}) = N(\beta_n)\left\lfloor\frac{(\alpha_n,\beta_n)_\times}{N(\beta_n)}\right\rceil\left(2\frac{(\alpha_n,\beta_n)_\times}{N(\beta_n)} - \left\lfloor\frac{(\alpha_n,\beta_n)_\times}{N(\beta_n)}\right\rceil\right) \geq 0.
$$

Note that the difference can be either less than, or greater than, or equal to $N(\beta_n)$. Consider the first case. Then

$$
\left|2\frac{(\alpha_n,\beta_n)_\times}{N(\beta_n)} - \left\lfloor\frac{(\alpha_n,\beta_n)_\times}{N(\beta_n)}\right\rceil\right| < 1.
$$

Note that for the function $f(x) = |2x - \lfloor x \rceil|$ the inequality $0 \leq f(x) < 1$ holds only for $-1 < x < 1$. Therefore, $-1 < (\alpha_n, \beta_n)_\times / N(\beta_n) < 1$, and $\lfloor (\alpha_n, \beta_n)_\times / N(\beta_n) \rceil = \pm 1$, depending on the sign of $(\alpha_n, \beta_n)_\times$. We do not consider the case $\lfloor (\alpha_n, \beta_n)_\times / N(\beta_n) \rceil  = 0$, because it indicates the termination of our algorithm. Suppose that $(\alpha_n, \beta_n)_\times \geq 0$ (we can prove the opposite case analogously). Then

$$
\alpha_{n + 1} = \beta_n, \,\,\,\,\, \beta_{n + 1} = \alpha_n -  \left\lfloor\frac{(\alpha_n,\beta_n)_\times}{N(\beta_n)}\right\rceil \beta_n = \alpha_n - \beta_n;
$$

$$
\alpha_{n + 2} = \beta_{n + 1} = \alpha_n - \beta_n, \,\,\,\,\, \beta_{n + 2} = \alpha_{n+1} - \left\lfloor\frac{(\alpha_{n+1},\beta_{n+1})_\times}{N(\beta_{n+1})}\right\rceil \beta_{n+1}.
$$

Let us simplify the integer occurring in $\beta_{n+2}$:

$$
\left\lfloor\frac{(\alpha_{n+1},\beta_{n+1})_\times}{N(\beta_{n+1})}\right\rceil = \left\lfloor\frac{(\beta_n, \alpha_n - \beta_n)_\times}{N(\alpha_n) - 2(\alpha_n,\beta_n)_\times + N(\beta_n)} + \frac{1}{2}\right\rfloor =
$$

$$
= \left\lfloor \frac{N(\alpha_n) - N(\beta_n)}{2\left(N(\alpha_n) - 2(\alpha_n,\beta_n)_\times + N(\beta_n)\right)} \right\rfloor = \left\lfloor\frac{1}{2} - \frac{N(\beta_n) - (\alpha_n, \beta_n)_\times}{N(\alpha_n) -2(\alpha_n,\beta_n)_\times + N(\beta_n)}\right\rfloor.
$$

If we now look closer at the difference under the floor sign, we notice that it actually varies from 0 to $1/2$, excluding $1/2$, because $N(\alpha_n) > N(\beta_n)$. This allows us to conclude that the integer we have been computing is actually zero, and therefore $\beta_{n + 2} = \alpha_{n + 1} = \beta_n$. Summarizing, we have

$$
\begin{array}{l l l}
\alpha_n, & \alpha_{n+1} = \beta_n, & \alpha_{n+2} = \alpha_n - \beta_n;\\
\beta_n, & \beta_{n+1} = \alpha_n - \beta_n, & \beta_{n + 2} = \beta_n.
\end{array}
$$

We conclude that whenever the norm of $N(\alpha_n)$ gets reduced by less than $N(\beta_n)$, our algorithm will terminate either on the $(n+1)$-st, or on the $(n+2)$-nd iteration.

\par Now, consider the second case, when $N(\alpha_n) - N(\beta_{n+1}) \geq N(\beta_n)$. Let us assume the worst case, when on each iteration of the algorithm the norm of $N(\alpha_n)$ gets reduced exactly by $N(\beta_n)$. Denote $A_n = N(\alpha_n)$, $B_n = N(\beta_n)$, and let $f_n$ for natural $n$ denote the $n$-th element of the Fibonacci sequence, i.e. $f_1 = 1$, $f_2 = 1$, $f_{n + 2} = f_{n+1} + f_n$. Then on the first six iterations, our algorithm produces the result, demonstrated in the table below. Note that the \nth{4} column demonstrates which inequality has to be satisfied in order for the algorithm to proceed to the next iteration, and the \nth{5} column demonstrates an interval, which $A/B$ belongs to if two consecutive iterations did not result in a termination of the algorithm.

$$
\begin{array}{c | c | c | c | c}
n & A_n & B_n & A_n > B_n & \textnormal{Interval for A/B}\\
\hline
1 & A & B & f_1 A > f_2 B & \multirow{2}{*}{$\frac{f_3}{f_2} \geq \frac{A}{B} \geq \frac{f_2}{f_1}$}\\
2 & B & A - B & f_3 B > f_2 A\\
\hline
3 & f_1 A - f_2 B & f_3 B - f_2 A & f_3 A > f_4 B& \multirow{2}{*}{$\frac{f_5}{f_4} \geq \frac{A}{B} \geq \frac{f_4}{f_3}$}\\
4 & f_3 B - f_2 A & f_3 A - f_4 B & f_5 B > f_4 A\\
\hline
5 & f_3 A - f_4 B & f_5 B - f_4 A & f_5 A > f_6 B & \multirow{2}{*}{$\frac{f_7}{f_6} \geq \frac{A}{B} \geq \frac{f_6}{f_5}$}\\
6 & f_5 B - f_4 A & f_5 A - f_6 B & f_7 B > f_6 A\\
\hline
& ... & & ... &
\end{array}
$$

\begin{center}
\small
\par Table 1
\end{center}

\par In general, it is not hard to show that, depending on the parity of $n$, for $n > 2$ we get

$$
\begin{array}{c  c  c}
n \,\,\, \textnormal{is odd:} & \alpha_n = f_{n - 2}A - f_{n - 1}B, & \beta_n = f_n B - f_{n - 1}A;\\
n \,\,\, \textnormal{is even:} & \alpha_n = f_{n - 1}B - f_{n - 2}A, & \beta_n = f_{n - 1} A - f_n B.
\end{array}
$$

\par Using the famous relation $f_{n-1}f_{n+1}-f_n^2 = (-1)^n$, which holds for any $n \geq 2$, we can demonstrate that an interval, defined on the iteration $2m+2$, is a subset of an interval, defined on the iteration $2m$:

$$
\left[\frac{f_{2m+2}}{f_{2m+1}} ; \frac{f_{2m + 3}}{f_{2m + 2}}\right] \subset \left[\frac{f_{2m}}{f_{2m - 1}} ; \frac{f_{2m+1}}{f_{2m}}\right].
$$

\par So as we can see, our interval gets shorter with each iteration, and our algorithm terminates whenever $A/B$ falls out of the interval.

\par But we can go even further: recall that $\lim\limits_{n \rightarrow \infty} f_{n+1}/f_n = \varphi$, where $\varphi$, known as the \emph{golden ratio}, is equal to $(1+\sqrt 5)/2$. Since this sequence converges, any subsequence of it must converge as well, which implies that $\{f_{2m}/f_{2m-1}\}_{m =1 }^\infty$ and $\left\{f_{2m+1}/f_{2m}\right\}_{m = 1}^\infty$ both tend to $\varphi$. But the first sequence is monotonously increasing, while the second one is monotonously decreasing. It means that each interval of the form $[f_{2m} / f_{2m - 1} ; f_{2m+1} / f_{2m}]$ must contain $\varphi$. As a consequence, we can see that the closer $A/B$ to $\varphi$, the more iterations our algorithm has to perform.

\par There are two separate cases arise at this point: when $A/B$ is less than, or greater than $\varphi$. If $A / B < \varphi$, then  $A / B < \varphi < f_{2m+1}/f_{2m}$ is true for any $m$, which means that the algorithm cannot terminate at the even iteration. Analogously, if $A/B > \varphi$, then $f_{2m}/f_{2m-1} < \varphi < A/B$, and the algorithm cannot terminate at the odd iteration.

\par Note that we expect out algorithm to terminate when $A_k \geq B_k$ for some $k$. Assume that $A/B > \varphi$, which means that $k = 2m$ for some positive integer $m$. Recall the exact formula for the $n$-th element of the Fibonacci sequence:

$$
f_n = \frac{\varphi^n - (1 - \varphi)^n}{\sqrt 5}.
$$

Then for $n$, satisfying $A / B \geq f_{2m + 1} / f_{2m}$, we have

$$
\frac{A}{B} \geq \frac{\varphi^{2m + 1} - (1 - \varphi)^{2m + 1}}{\varphi^{2m} - (1 - \varphi)^{2m}},
$$

\par If we now take the logarithm on both parts of the inequality above, we shall get

$$
(2m + 1)\ln \varphi + \ln\left(1 + \left(\frac{\varphi - 1}{\varphi}\right)^{2m + 1}\right) - 2m \ln \varphi - \ln\left(1 - \left(\frac{\varphi - 1}{\varphi}\right)^{2m}\right) \leq \ln \frac{A}{B};
$$

Since $\ln(1 + x) \approx x$ for $x \ll 1$, and $((\varphi - 1)/\varphi)^{2m}$ tends to zero as $m$ goes to infinity, then for big values of $m$ the following inequality holds:

$$
\left(\frac{\varphi - 1}{\varphi}\right)^{2m + 1} + \left(\frac{\varphi - 1}{\varphi}\right)^{2m} \leq \ln\frac{A}{B\varphi},
$$

which, in turn, results in

$$
\left(\frac{\varphi - 1}{\varphi}\right)^{2m} \leq \frac{\varphi}{2\varphi - 1}\ln\frac{A}{B\varphi}.
$$

If we now take the logarithm on both sides of the inequality above, we shall obtain

$$
k = 2m \geq \frac{\ln\frac{2\varphi - 1}{\varphi \ln (A/(B\varphi))}}{\ln\frac{\varphi}{\varphi - 1}}.
$$

Let

$$
a = \left(\ln\frac{\varphi}{\varphi - 1}\right)^{-1} \approx 1.039, \,\,\,\,\, b = \frac{\ln\frac{2\varphi - 1}{\varphi}}{\ln \frac{\varphi}{\varphi - 1}} \approx 0.336.
$$

Then

$$
k \geq -a\ln\ln\frac{A}{B\varphi} + b,
$$

which means that for any even $k$, which satisfies the inequality above, the relation $A/B \geq f_{k+1}/f_k$ holds. In other words, the closest even integer greater than $-a\ln\ln A/(B\varphi) + b$ will be the least upper bound for $k$, which allows us to conclude that $k(A, B) \in O\left(\log\left|\log \frac{A}{B\varphi}\right|\right)$ as $A/B \rightarrow \varphi$. Here we used the module sign because in case if $A/B < \varphi$, the value of $\ln A / (B\varphi)$ is negative. We also expect the value $\ln\left|\ln\frac{A}{B\varphi}\right|$ to be negative.

\par Now, let's see how bad our algorithm will perform, if we will make $A/B$ as close to $\varphi$ as possible. If $A / B > \varphi$, this happens when $B = \lfloor A / \varphi \rfloor$. In this case,

$$
\ln\ln\frac{A / \varphi}{\lfloor A / \varphi \rfloor} = \ln\ln\left(1 + \frac{\{A / \varphi\}}{\lfloor A / \varphi \rfloor}\right) \approx \ln\frac{\{A / \varphi\}}{\lfloor A / \varphi \rfloor} = -\ln \frac{\lfloor A / \varphi \rfloor}{\{A / \varphi\}}.
$$

We conclude that in this case

$$
k \geq a\ln\frac{\lfloor A / \varphi \rfloor}{\{ A / \varphi\}} + b,
$$

and therefore $k(A) \in O\left(\log \frac{A}{\{A / \varphi\}}\right)$ as $A$ approaches infinity.
\end{proof}

\begin{rem}
From the theorem \ref{thm:reduce_bound}, it follows that the algorithm \ref{alg:reduction} executes in $O\left(\log\log\frac{N(\upsilon_1)}{N(\upsilon)\varphi}\right)$, assuming that $N(\upsilon_1) \geq N(\upsilon)$. If $N(\upsilon_1)/N(\upsilon)$ approximates $\varphi$ in the best possible manner, i.e. $N(\upsilon) = \lfloor N(\upsilon_1) / \varphi \rfloor$, then the algorithm works in $O\left(\log\frac{N(\upsilon_1)}{\{N(\upsilon_1) / \varphi\}}\right)$.
\end{rem}

\section{Generalizing the result of Fermat} \label{sec:fermat}

Recall the Fermat factorization method: suppose that we know a couple of different representations of a positive integer $m$ as a sum of two squares:

$$
m = x_0^2 + y_0^2 = x_1^2 + y_1^2, \,\,\, x_0 \geq y_0 \geq 0, \,\,\, x_1 \geq y_1 \geq 0, \,\,\, x_0 > x_1.
$$

Then $1 < \gcd(x_0y_1 - y_0x_1, m) < m$ \cite[Sec. 5.6.2]{crandall}. Same formula holds if we know a couple of different representations of $m$ as the sum of a square and a doubled square: $m = x_0^2 + 2y_0^2 = x_1^2 + 2y_1^2$.

\par The result of Fermat is a special case of the theory, introduced in this paper. Consider a quadratic order $\ord(\mu)$ with $\mu = xi + yj + zk$, where $N(\mu) = x_0^2 + y_0^2 + z_0^2$ is composite. Then by the theorem \ref{thm:ambiguous}, the pseudo generator of an ambiguous ideal $\Rideal \rho$ has to divide $(2/r)\mu$ from the right, and therefore $N(\rho)$ has to divide $(2/r)^2m$. Since $\rho$ satisfies the relation (\ref{eq:rhomu}) for some $\mu' = x_1i + y_1j + z_1k$ of the same norm as $\mu$, it is evident that for any three squares representation of an integer there exists another three squares representation, which allows us to factor $(2/r)^2 m$. The following couple of theorems demonstrate that for quaternions $\mu = x i + yj$ and $\mu = xi + yj + yk$ and an ambiguous ideal $\Rideal \rho$, $\mu'$ has to be of a certain form.

\begin{thm} \label{thm:two_squares}
Consider a quadratic order $\ord(\mu)$, where $\mu = x_0 i + y_0 j + z_0 k$, and exactly one of the coefficients $x_0$, $y_0$, $z_0$ is equal to zero. The equivalence class $[\Rideal \rho]$ is ambiguous if and only if for the quaternion $\mu' = \rho\mu\rho^{-1} = x_1 i + y_1 j + z_1 k$ exactly one of the coefficients $x_1$, $y_1$, $z_1$ is zero.
\end{thm}

\begin{proof}
($\Rightarrow$)\footnote{The necessary condition was proved by Mikhail N. Kubensky.} Let $\Rideal \rho$ be an ambiguous ideal, which belongs to the order $\ord(\mu)$, where $\mu = x_0 i + y_0 j$ and $m = N(\mu) = x_0^2 + y_0^2$. Note that in here we have picked $\mu$ so that $z_0 = 0$, but in general we can zero out any of the coefficients by replacing $\ord(\mu)$ with a suitable quadratic order of the form $\ord(\varepsilon \mu \overline{\varepsilon})$ for some unit $\varepsilon$.

\par By the theorem \ref{thm:ambiguous}, $\rho$ has to divide $(2/r)\mu = 2\mu$ from the right, i.e. $2\mu = \rho'\rho$ for some integral quaternion $\rho'$. Then $2\mu' = 2(x_1i + y_1j + z_1k) = \rho(2\mu)\rho^{-1} = \rho\rho'$. Let us demonstrate that $z_1 = 0$.

\par First of all, notice that $2\mu i = -x_0 - y_0 k$ is an element of the ring of Gaussian integers $\RZ[k]$. Since $\RZ[k]$ is a unique factorization domain \cite[Chap. 1]{mollin}, the quaternion $2\mu i$ of norm $2m = N(\rho')N(\rho)$ can be uniquely represented as the product of two integral quaternions $r = p_0 + q_0k$ and $r' = p_1 + q_1 k$, which are also in $\RZ[k]$, and $N(r) = N(\rho')$, $N(r') = N(\rho')$:

$$
2\mu i = \rho'\rho i = r'r.
$$

Since $2\mu = r'(-ri)$, put

$$
\rho = -ri = -p_0 i + q_0j, \,\,\, \rho' = r' = p_1 + q_1 k.
$$

Then $\mu' = \rho\rho'$ has $z_1 = 0$.

\par ($\Leftarrow$) Let $\mu = x_0 i + y_0 j$, $\mu' = x_1 i + y_1 j$, and $N(\mu) = N(\mu')$. Once again, we have picked our $\mu$ and $\mu'$ so that $z_0 = z_1 = 0$, but this is not necessary, as we can always find suitable quaternions $\varepsilon_0 \mu \overline{\varepsilon_0}$ or $\varepsilon_1 \mu' \overline{\varepsilon_1}$ for some units $\varepsilon_0$, $\varepsilon_1$, which would make $z_0 = z_1 = 0$.

\par By the theorem \ref{thm:ideal_generated_by_orders}, an integral quaternion $\rho' = (\mu + \mu')/d$, where $d$ is the greatest integral divisor of $\mu + \mu'$, is a right pseudo generator of a primitive ideal $[N(\rho), b + \mu] \subseteq \ord(\mu)$ for some integer $b$. Since $\rho'$ is a right pseudo generator of an ideal, then $\rho = i\rho'$ is a pseudo generator of the same ideal:

$$
\rho = i\rho' = \frac{-(x_0 + x_1) + (y_0 + y_1)k}{d}.
$$

Obviously, we have $(\mu,\rho) = 0$, which means that the relation (\ref{eq:main_relation_rho}) takes the form

$$
\Re(\rho)b = t\Re(\xi)\frac{N(\rho)}{t},
$$

where $t$ is defined as in step 2 of the algorithm \ref{alg:restore}. Since $\gcd(\Re(\rho), N(\rho)/t) = 1$ (otherwise $\rho$ would not be primitive), we have $b = kN(\rho) / t$ for an integer $k = t\Re(\xi)/\Re(\rho)$. It implies that the $\RZ$-basis of $\Rideal \rho$ takes the form $[N(\rho), kN(\rho)/t + \mu]$. In turn, this ideal is equivalent to $[N(\rho), \frac{t - 1}{t}N(\rho) + \mu]$, which is ambiguous.
\end{proof}

\begin{thm}
Consider a quadratic order $\ord(\mu)$, where $\mu = x_0 i + y_0 j + z_0 k$, and exactly two of the coefficients $x_0$, $y_0$, $z_0$ are equal to each other in their absolute value. The equivalence class $[\Rideal \rho]$ is ambiguous if and only if for the quaternion $\mu' = \rho\mu\rho^{-1} = x_1 i + y_1 j + z_1 k$ exactly two of the coefficients $x_1$, $y_1$, $z_1$ are equal to each other in their absolute value.
\end{thm}

\begin{proof}
($\Rightarrow$) If we have $\mu = x_0i + y_0(j + k)$, then $\mu i = -x_0 + y_0(j - k)$ belongs to the ring $\RZ[j - k]$, which is isomorphic to $\RZ[\sqrt{-2}]$. Since $\RZ[\sqrt{-2}]$ is a unique factorization domain, the same reasoning as in the proof of the theorem \ref{thm:two_squares} can be applied to demonstrate the desirable result.

\par ($\Leftarrow$) The sufficient condition can be proved using almost the same reasoning as in the theorem \ref{thm:two_squares}. The only contrast between two proves would be in the difference of initial quaternions $\mu$ and $\mu'$.
\end{proof}

\section{Some relations, which characterize an ideal}

\par In this section, we present three formulas, which characterize a primitive ideal $\Rideal \rho = [a, b + \omega]$ in the quadratic order $\ord(\mu)$. The first formula that we are about to introduce has the same form as (\ref{eq:main_relation_rho}), except that $\rho$ gets replaced by $\xi$, which satisfies $b + \omega = \xi\rho$. To prove it, we first need to demonstrate the following proposition.

\begin{prop} \label{prop:equal_scalar_products_xi}
Consider $\Rideal \rho = [a, b + \omega]$ in the quadratic order $\ord(\mu)$. Then $(\omega,\xi) = (\omega',\xi)$, where $\omega' = \rho\omega\rho^{-1}$, and $\xi$ satisfies the relation $b + \omega = \xi\rho$.
\end{prop}

\begin{proof}
Consider two equalities: $\xi(b + \omega) = \xi^2\rho$ and $(b + \omega')\xi = \rho\xi^2$. Since every quaternion satisfies the relation (\ref{eq:quat_relation}), we have

$$
\xi(b + \omega) = \xi^2\rho = \left( 2\Re(\xi)\xi - N(\xi)\right)\rho = 2\Re(\xi)(b + \omega) - N(\xi)\rho;
$$

$$
(b + \omega')\xi = \rho\xi^2 = \rho\left(2\Re(\xi)\xi - N(\xi)\right) = 2\Re(\xi)(b + \omega') - N(\xi)\rho.
$$

Let us subtract the second equation from the first one:

$$
\xi(b + \omega) - (b + \omega')\xi = 2\Re(\xi)(\omega - \omega');
$$

$$
\xi\omega - \omega'\xi = 2\Re(\xi)(\omega - \omega').
$$

Consider the quaternion on the left side of the equation:

$$
\xi\omega - \omega'\xi = \left(\Re(\xi) + \vec\xi\right)(\Re(\omega) + \vec\omega) - (\Re(\omega') + \vec{\omega'})\left(\Re(\xi) + \vec\xi\right) = 
$$

$$
= \Re(\xi)\Re(\omega) - \Re(\xi)\Re(\omega') + \left(\vec\xi\vec\omega - \vec{\omega'}\vec\xi\right) + \left(\Re(\xi)\vec\omega + \Re(\omega)\vec\xi - \Re(\xi)\vec{\omega'} - \Re(\omega')\vec\xi\right).
$$

Since $\Re(\omega) = \Re(\omega')$,

$$
\xi\omega - \omega'\xi = \left(\vec\xi\vec\omega - \vec{\omega'}\vec\xi\right) + \Re(\xi)(\vec\omega - \vec{\omega'}) = 2\Re(\xi)(\omega - \omega').
$$

Applying the relation (\ref{eq:product}), we obtain

$$
-(\omega,\xi) + [\xi,\omega] + (\omega',\xi) - [\omega',\xi] + \Re(\xi)\left(\vec\omega - \vec{\omega'}\right) = 2\Re(\xi)\left(\omega - \omega'\right).
$$

Since the quaternion on the right has a zero real part, if we now take the real part on both sides of the equation above, we get $(\omega,\xi) = (\omega',\xi)$.
\end{proof}

\begin{prop} \label{prop:main_prop_xi}
If a primitive ideal $[a, b + \omega] = \Rideal \rho$ belongs to the quadratic order $\ord(\mu)$, then

\begin{equation} \label{eq:main_relation_xi}
(\omega, \xi) +\Re(\xi)\left(b + \frac{r - 1}{r}\right) = \Re(\rho)N(\xi).
\end{equation}

where $\xi$ is an integral quaternion, which satisfies the relation $b + \omega = \xi\rho$, and an integer $r$ is defined as in (\ref{eq:omega}).
\end{prop}

\begin{proof}
If we use the fact from proposition \ref{prop:equal_scalar_products_xi} that $(\omega,\xi) = (\omega,\xi')$, the relation (\ref{eq:main_relation_xi}) can be shown in the same manner as in the proof of the proposition \ref{prop:main_prop}.
\end{proof}

Next two relations follow from the fact that an integral quaternion $\omega + \omega'$ will always be a root of the the equation $\rho\omega = \omega'\rho$. By the proposition \ref{prop:equal_scalar_products_rho}, $(\omega + \omega', \omega) = (\omega + \omega', \omega')$. It follows that

$$
\left(\frac{\omega + \omega'}{2}, \omega\right) = \left(\frac{\omega + \omega'}{2}, \frac{\omega + \omega'}{2}\right) = \frac{m + (\mu,\mu')}{r^2},
$$

since $(\omega,\omega) = (\omega',\omega') = m / r^2$, where $m = N(\mu) = N(\mu')$. If we now apply the \nth{2} relation from the system (\ref{eq:system}), we shall obtain

$$
\left(\frac{\omega + \omega'}{2}, \frac{\omega + \omega'}{2}\right) = \left(\Re(\xi)\vec\rho + \Re(\rho)\vec\xi+ \frac{r - 1}{r}, \Re(\xi)\vec\rho + \Re(\rho)\vec\xi+ \frac{r - 1}{r}\right) =
$$

$$
= \Re(\rho)^2(\xi,\xi) + 2\Re(\rho)\Re(\xi)(\rho,\xi) + \Re(\xi)^2(\rho,\rho);
$$

$$
\left(\frac{\omega + \omega'}{2}, \omega\right) = \left(\Re(\xi)\vec\rho + \Re(\rho)\vec\xi + \frac{r-1}{r}, \omega\right) = 
$$

$$
= \Re(\xi)(\omega,\rho) + \Re(\rho)(\omega,\xi) = \frac{\Re(\xi)(\mu,\rho) + \Re(\rho)(\mu,\xi)}{r}.
$$

In the end, we obtain the following couple of relations:

\begin{equation}
\frac{m + (\mu,\mu')}{r} = \Re(\xi)(\mu,\rho) + \Re(\rho)(\mu,\xi);
\end{equation}

\begin{equation}
\frac{m + (\mu,\mu')}{r^2} = \Re(\rho)^2(\xi,\xi) + 2\Re(\rho)\Re(\xi)(\rho,\xi) + \Re(\xi)^2(\rho,\rho).
\end{equation}

\section{Number of ambiguous classes, characterized by the equation $\rho\mu = -\mu\rho$} \label{sec:count_ambiguous}

From the theorem \ref{thm:ideal_generated_by_orders}, we know that each ideal gets generated by a couple of quadratic orders. Among these ideals, some of them get generated by orders $\ord(\mu)$ and $\ord(-\mu)$. In our research, we were interested in the following question: how many nontrivial \emph{ambiguous} ideals get generated by such quadratic orders? To answer it, we have written a program, which counts up to some bound $N$ how many integers, which satisfy certain parameters mentioned below, have at least one ambiguous ideal of that form.

\par First of all, we have to emphasize that it is not always the case that $\ord(\mu)$ and $\ord(-\mu)$ are not equivalent. Normally, if $m = x^2 + y^2 + z^2$ and $x$, $y$, $z$ are all distinct and non-zero, out of a single three squares representation we can produce 48 different, including the original one. These representations can be obtained by switching signs and by permuting $x$, $y$, $z$. On the other hand, if $\mu = xi + yj + zk$, there exist exactly 12 equivalent quadratic orders of the form $\ord(\varepsilon \mu \overline{\varepsilon})$. Now it is not hard to observe that if $x$, $y$, $z$ are different and non-zero, quaternions $\mu$, $-\mu$, $\tilde \mu$ and $-\tilde\mu$, where $\tilde\mu = xi + zj + yk$, will form quadratic orders, which are not equivalent to each other.

\par There are two special cases arise at this point: if $\mu = xi + yj + zk$ and $y = z$, then $\tilde\mu = \mu$, i.e. quadratic orders, generated by these quaternions, are equivalent. Analogous observation can be made when $x = z$ or $x = y$. If $\mu = xi + yj + zk$ and $z = 0$, then $-\mu = k\mu\overline{k}$, which means that $\ord(\mu)$ and $\ord(-\mu)$ are equivalent to each other. Once again, analogous observation can be made when $x = 0$ or $y = 0$.

\par As a result, we have decided to avoid two squares representations, for they will always generate trivial ambiguous ideals. We can also argue that the case $\mu = xi + yj + zk$ when $y = z$ will also produce a trivial result. Recall the theorem \ref{thm:solutions}. In its proof, we have demonstrated how to compute the $\RZ$-module $[\upsilon, \upsilon_1]$, which fully characterizes the set of solutions to the equation $\rho\mu = -\mu\rho$. If we would follow the procedure described in that proof, we could notice that $\upsilon = j - k$. Because $\ord(\mu)$ and $\ord(-\mu)$ are not equivalent, the norm $N(\upsilon) = 2$ is the smallest, which means that $\Rideal \upsilon$ is reduced. When $m \equiv 3$ (mod 8), from the theorem \ref{thm:ideal_generated_by_orders} we know that such an ideal does not belong to $\ord(\mu)$. In the opposite case, from the theorem \ref{thm:ambiguous} we know that $N(\upsilon) = 2$ has to divide $(2/r^2)m = 4m$, which obviously corresponds to the trivial case. This is why we ignore those $\mu = xi + yj + zk$ with $y = z$, as well as the cases when $x = y$ or $x =  z$, which produce the trivial result for the same reason.

\par The table below demonstrates our computational data. In our experiment, we were considering only those integers, which are a) squarefree; b) not prime; and c) not congruent to 7 mod 8 (since these integers cannot be represented as a sum of three squares). Let us denote the set of those integers as $\Sigma$. Among the elements of $\Sigma$, our algorithm was looking for those $m$, which have at least one quaternion $\mu = xi + yj + zk$ with the norm $m = N(\mu)$, s.t. quadratic orders $\ord(\mu)$ and $\ord(-\mu)$ generate a non-trivial ambiguous ideal. We denote this set as $A$. The \nth{4} column demonstrates what is the percentage of elements of $A$ in respect to the number of elements of $\Sigma$. The \nth{5} column shows an integer less than $N$, which has the maximal number $M$ of non-trivial ambiguous ideals, generated by $\rho\mu = -\mu\rho$. The number $M$ is indicated in brackets. Finally, the \nth{6} column gives an example of an element $m = x^2 + y^2 + z^2$ in $A$ and a three squares representation ($x$,$y$,$z$), which generates an ambiguous ideal, while the \nth{7} column gives an example of an integer, which is in $\Sigma \setminus A$.

\begin{center}
\footnotesize
\begin{tabular}{| c | c | c | c | c | c | c |}
\hline
$N$ & \# $\Sigma$ & \# $A$ & \% & Max \# &  E.g. $A$ & E.g. $\Sigma \setminus A$\\
\hline
\hline
$10^3$ & 379 & 151 & 39.84 & 645 (4) & 21 (4,2,1) & 6\\
\hline
$10^4$ & 4145 & 1853 & 44.70 & 2310 (8) & 1001 (26,15,10) & 1002\\
\hline
$10^5$ & 43464 & 20584 & 47.36 & 90321 (32) & 10001 (98,19,6) & 10003\\
\hline
$10^6$ & 447767 & 220308 & 49.20 & 899745 (64) & 100002 (281,121,80) & 100001\\
\hline
$10^7$ & 4567729 & 2302087 & 50.40 & 899745 (64) & 1000001 (770,630,101) & 1000006\\
\hline
\end{tabular}

\small
\par Table 2
\end{center}

There are two interesting tendencies that we can notice in this table. First of all, the overall percentage of numbers which have ambiguous ideals generated by $\rho\mu = -\mu\rho$ is growing. Second, each integer less than $N$, which has a maximal number $M$ of such an ambiguous classes, has $M$ equal to some power of $2$. We can look closer at this sequence, and ask: is there a relation between the number of these ambiguous classes, and the ideal class group? The following table demonstrates that such a correlation seem to exist:

\begin{center}
\begin{tabular}{| c | c | c | c | l |}
\hline
$N$ & $M$ & $\Delta$ & $h_\Delta$ & $Cl_\Delta$\\
\hline
\hline
21 & 1 & -84 & 4 & 2 2\\
\hline
105 & 2 & -420 & 8 & 2 2 2\\
\hline
645 & 4 & -2580 & 16 & 4 2 2\\
\hline
2310 & 8 & -9240 & 32 & 4 2 2 2\\
\hline
10605 & 16 & -42420 & 64 & 4 4 2 2\\
\hline
90321 & 32 & -361284 & 256 & 8 4 2 2 2\\
\hline
899745 & 64 & -3598980 & 512 & 8 4 2 2 2 2\\
\hline
\end{tabular}

\small
\par Table 3
\normalsize
\end{center}

\par Here $\Delta$ denotes the discriminant of a maximal quadratic order of an imaginary number field $\QF{-N}$, $h_\Delta$ is the class number, and $Cl_\Delta$ denotes the ideal class group. Following the definition \cite[Sec. 7.1]{jacobson}, we write the structure of each ideal class group as a direct product of cyclic subgroups, i.e.

$$
Cl_\Delta \cong C(m_1) \times \ldots \times C(m_s),
$$

where the positive integers $m_1, \ldots, m_s$ (which are presented in the \nth{5} column) satisfy $m_1 \geq 1$, $m_{j + 1} \,\, | \,\, m_j$ for $1 \leq j < s$, and $C(x)$ denotes the cyclic group of order $x$. As we can notice, for each $N$, the corresponding value of $h_\Delta$ is a power of 2, and the class group is non-cyclic.

\section{A new approach to find a divisor of a class number}

Fix the squarefree integer $m \not \equiv 7$ (mod 8), and consider a maximal quadratic order $\ord$ with discriminant $\Delta$ of an imaginary number field $\QF{-m}$. If the class number $h_\Delta = ef$ for some positive integers $e$, $f$, there exists a subgroup $A$ of the ideal class group $Cl_\Delta$, which consists of $f$ equivalence classes with reduced ideals $\mathfrak a_1, \mathfrak a_2, \ldots, \mathfrak a_f$ as their representatives. If we now generate an integral quaternion $\mu_1$ s.t. $\mu_1^2 = -m$ and replace $\sqrt{-m}$ in the $\RZ$-basis of each ideal $\mathfrak a_i$ (where $1 \leq i \leq f$) with $\mu_1$, then we will ``move'' the subgroup $A$ from $\ord$ to $\ord(\mu_1)$. By the theorem \ref{thm:ideal_generated_by_orders}, each ideal $\mathfrak a_i$ in $\ord(\mu_1)$ generates some quadratic order $\ord(\mu_i)$ by means of relation (\ref{eq:rhomu}). Moreover, by the theorem \ref{thm:equivalence_class}, we know that all $\mu_1, \ldots, \mu_f$ are distinct, in a sense that no two integral quaternions from this set generate equivalent quadratic orders. Hence, there exists a bijective correspondence between $\mathfrak a_1, \ldots, \mathfrak a_f$ and $f$ distinct equivalence classes of quadratic orders with $\ord(\mu_1), \ldots, \ord(\mu_f)$ taken as their representatives (assuming that initially all ideals belong to $\ord(\mu_1)$) \cite[\textsection 18]{ven4}.

\par Now, pick an ideal $\mathfrak a \in A$ of order $f > 1$, i.e. $f$ is the smallest positive integer s.t. $\mathfrak [a]^f = 1$ holds\footnote{Here $[\mathfrak a]$ denotes the equivalence class of $\mathfrak a$, and $1$ denotes the unit in $Cl_\Delta$, which is simply $[\ord]$.}. Then $\mathfrak a$ in $\ord(\mu)$ generates a new quadratic order, $\ord(\mu')$. Replacing $\mu$ in the $\RZ$-basis of $\mathfrak a$ with $\mu'$, we will ``move'' $\mathfrak a$ to $\ord(\mu')$. Once again, $\mathfrak a$ in $\ord(\mu')$ generates a new quadratic order $\ord(\mu'')$. If we now proceed in this fashion, at some point our $\mathfrak a$ will return to $\ord(\mu)$, generating a cycle of length $f$ \cite[\textsection 19]{ven4}. This cycle will run through all of the $\ord(\mu_1), \ldots, \ord(\mu_f)$, which were mentioned in the previous paragraph.

\par Consider a problem of finding the length $f$ of a cycle, which is equivalent to the problem of finding the order of $\mathfrak a$. Observe that for any quaternion $\mu_i$ for $1 \leq i \leq f$, a quadratic order $\ord(\mu_i)$ is equivalent to some quadratic order $\ord(\mu_i')$, where $\mu_i' = \varepsilon\mu_i\overline{\varepsilon}$ for some unit $\varepsilon$ has either all non-negative, or all non-positive coefficients. We shall call $\ord(\mu)$ \emph{positive} if the first case holds, and \emph{negative} if the second case holds\footnote{Note that for the quaternion $\mu = xi + yj + zk$, where exactly one of the coefficients $x$, $y$, $z$ is equal to zero, the order $\ord(\mu)$ is both positive and negative.}.

\begin{defin}
Consider an ideal $\mathfrak a$ in $\ord(\mu_1)$, which generates a cycle $\ord(\mu_1), \ldots, \ord(\mu_f)$. Let $\mathfrak a$ satisfy the \emph{separation property} if each $\ord(\mu_i)$ is positive (negative) when $1 \leq i \leq \lfloor f / 2 \rfloor$, and negative (positive) when $\lfloor f / 2 \rfloor < i \leq f$. Then we call $\mathfrak a$ \emph{separated} in $\ord(\mu_1)$.
\end{defin}

\par It is not hard to demonstrate that if we know an ideal, which satisfies the separation property in some $\ord(\mu_1)$, then we can find the length $f$ of its cycle in polynomial time. The idea is to find borders between positive and negative quadratic orders. As you can notice, there exist exactly two of them: between $\ord(\mu_{\lfloor f / 2\rfloor})$ and $\ord(\mu_{\lfloor f / 2 \rfloor + 1})$, and between $\ord(\mu_f)$ and $\ord(\mu_1)$.

\par Suppose that $\mathfrak a$ is separated in a positive order $\ord(\mu_1)$. Then, using the right pseudo generator of $\mathfrak a$, we compute $\ord(\mu_2)$ generated by $\mathfrak a$ in $\ord(\mu_1)$, $\ord(\mu_4)$ generated by $\mathfrak a^2$ in $\ord(\mu_2)$, and so on, until for some $n$ we obtain a negative order $\ord(\mu_n)$ generated by $\mathfrak a^{n / 2}$ in $\ord(\mu_{n/2})$. We conclude that $n/2 < d \leq n$. Now we shall move backwards, and try to tighten this interval. We ``move'' our ideal $\mathfrak a$ to $\ord(\mu_n)$, and now consider its left pseudo generator, which will direct us towards $\ord(\mu_1)$. Once again, we compute $\ord(\mu_{n - 1})$ generated by $\mathfrak a$ in $\ord(\mu_n)$, $\ord(\mu_{n - 3})$ generated by $\mathfrak a^2$ in $\ord(\mu_{n-1})$, and so on, until for some $m$ we get a positive order $\ord(\mu_{n - m})$ generated by $\mathfrak a^{(m+1)/2}$ in $\ord(\mu_{n + 1 - (m+1)/2})$. We conclude that $n - m < d \leq n + 1 - (m+1)/2$. If we proceed in this fashion, then on each iteration an interval which $d$ belongs to will get smaller and smaller, until it reaches the size where we can check the sign of each order in the interval.

\par On the picture below, we introduce an example for $f = 12$. Here, our distance to the border was found in three steps, which are demonstrated by arrows. Initially, the only thing we know about $d$ is that $d \geq 1$. After three jumps to the left on the first step, we conclude that $3 < d \leq 7$. On the second step, after two jumps to the right, we conclude that $4 < d \leq 6$. Finally, on the third step we find $4 < d \leq 5$, which means that $d = 5$.

\begin{center}
\includegraphics[scale=0.35, trim = -1cm 1cm -50mm 1cm]{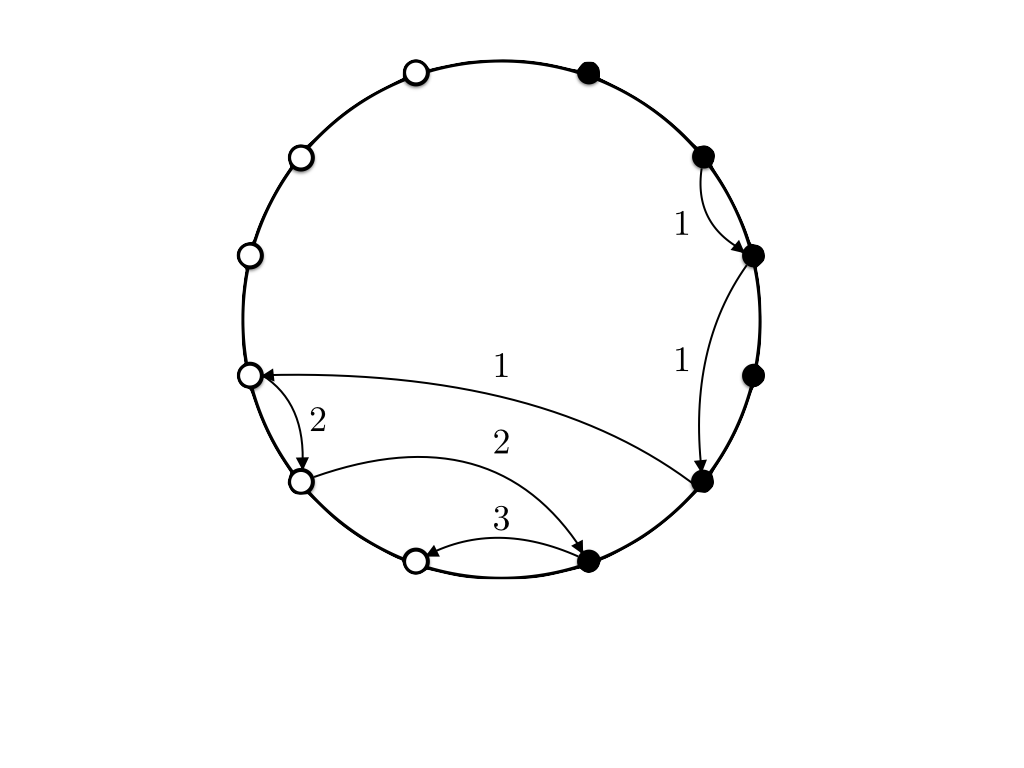}
\vspace{-2.5cm}
\par Pic. 1
\end{center}

\par Note that above we have found a distance from $\mu_1$ to only one of the borders. So what about the other border? We could find it analogously if from the very beginning we would start to compute quaternions, generated by $\mathfrak a$, $\mathfrak a^2$, etc. not with the right pseudo generator of $\mathfrak a$, but with the left pseudo generator.

\par There is one problem that may occur during our computations: for the reason that on each iteration we double the size of our step, it may happen that the step becomes bigger than the length of a cycle. If this happens, our procedure may result in wrong computations, or even in an infinite loop. Fortunately, there exists an upper bound on the size of a class number of an imaginary number field \cite[Sec. 5.10.1]{cohen3}. In particular, we know that

$$
h_\Delta < \frac{1}{\pi}\sqrt{|\Delta|}\ln |\Delta| \,\,\, \textnormal{if} \,\,\, \Delta < -4.
$$

As soon as our step becomes greater than the value above, we can simply pick a different $\ord(\mu_1)$, for example by replacing it with its right (or left) neighbour.

\par Let us give an example of an ideal, which satisfies the property described above. Let $\mathfrak a = \left[23, -2 + \sqrt{-893}\right]$ be an ideal in $\ord$ of $\QF{-893}$. If we now fix $\mu_1 = 29i + 4j + 6k$ (note that $893 = 29^2 + 4^2 + 6^2$), and then ``move'' $\mathfrak a$ to $\ord(\mu_1)$, i.e. $\mathfrak a = \left[23, -2 + \mu_1\right]$, then $\mathfrak a$ would generate the following cycle:

\setlength{\tabcolsep}{12pt}

\begin{center}
\begin{tabular}{l l}
$\mu_{14} = -28i -10j -3k$ & $\mu_1 = 29i + 4j + 6k$\\
$\mu_{13} = -21i -16j - 14k$ & $\mu_2 = 3i + 22j + 20k$\\
$\mu_{12} = -27i -10j -8k$ & $\mu_3 = 13i + 20j + 18k$\\
$\mu_{11} = -11i -24j - 14k$ & $\mu_4 = 13i + 18j + 20k$\\
$\mu_{10} = -11i -14j -24k$ & $\mu_5 = 3i + 20j + 22k$\\
$\mu_9 = -27i -8k -10k$ & $\mu_6 = 29i + 6j + 4k$\\
$\mu_8 = -21i -14j -16k$ & $\mu_7 = 28i + 10j + 3k$
\end{tabular}
\end{center}

Another separated ideal is $[18, 1 + \mu_1]$, where $\mu_1 = 42i + 14j + k$ ($1961 = 42^2 + 14^2 + 1^2$). It generates the following cycle:

\begin{center}
\begin{tabular}{l l}
$\mu_8 = -30i -10j -31k$ & $\mu_1 = 42i + 14j + k$\\
$\mu_7 = -14i -26j -33k$ & $\mu_2 = 18i + 26j + 31k$\\
$\mu_6 = -26i -14j -33k$ & $\mu_3 = 26i + 18j + 31k$\\
$\mu_5 = -10i - 30j - 31k$ & $\mu_4 = 14i + 42j + k$\\
\end{tabular}
\end{center}

If we consider the same quadratic order $\ord(\mu_1)$, and pick an ideal $\mathfrak b = [5, 2 + \mu_1]$, we can notice that $[\mathfrak b] = [\mathfrak a]^3$ does not satisfy the separation property:

\begin{center}
\begin{tabular}{l l}
$\mu_8 = -26i -14j -33k$ & $\mu_1 = 42i + 14j + k$\\
$\mu_7 = 26i + 18j + 31k$ & $\mu_2 = 14i + 42j + k$\\
$\mu_6 = -30i -10j -31k$ & $\mu_3 = -14i -26j - 33k$\\
$\mu_5 = -10i - 30j - 31k$ & $\mu_4 = 18i + 26j + 31k$\\
\end{tabular}
\end{center}

Many questions arise at this point. First of all, do separated ideals exist in $\ord(\mu)$  for any squarefree $N(\mu)$? Second, can we say anything meaningful about their quantity? Third, how hard is it to find a separated ideal? Probably, these ideals satisfy some other interesting properties that need to be explored. We leave all these questions open, hoping to dedicate more time to the study of separated ideals in the near future.

\section{Conclusion}

\par In this paper, we have introduced four algorithms, which allow us to manipulate ideals in the ring of integral quaternions using only their pseudo generators, without ever mentioning their $\RZ$-basis. In the future, we would like to extend our algorithms to the generalized quaternion algebras, possibly defined over a different field. We have also demonstrated the theorem, which shows that the Fermat's number factoring technique, which uses two different representations of an integer as a sum of two squares, is a special of case of the theory, presented in this paper. In general, if we have a single three squares representation for a composite integer $m = x_0^2 + y_0^2 + z_0^2$, there must exist another three squares representation $m = x_1^2 + y_1^2 + z_1^2$, which can help us to factor $m$. But the more interesting fact is that this is not always necessary, as sometimes a \emph{single} three squares representation can produce a non-trivial factor of $m$. This phenomenon was studied in the section \ref{sec:count_ambiguous}.

\par There are multiple questions left open at this point. What is the necessary, or, even better, sufficient condition for an integer $N$ to contain an ambiguous ideal, generated by the equation $\rho\mu = -\mu\rho$? Why each integer $N$, with a number $M$ of those ambiguous classes greater than any integer less than $N$, has $M$ being a power of 2? Why for each $N$ that satisfy this property, the class number $h_\Delta$ also turns out to be a power of 2? Finally, how can we find separated ideals within a certain ideal class group? We intend to explore all these questions in more depth in the future.

\subsection*{Acknowledgements}

This research was done for the author's Master's thesis, which was submitted to the faculty of higher mathematics of the St. Petersburg National Research University of Information Technologies, Mechanics and Optics in June, 2013. The author would like to thank his supervisor, Mikhail Nikolayevich Kubensky, for his wise guidance. The author is also grateful to Dr. Matthew Greenberg (University of Calgary, Canada) and Robert Bateman (L'Institut Polytechnique des Sciences Avanc\'ees, France) for their thoughtful comments and the paper revision.

\end{document}